\documentclass{amsart}

\usepackage{nccmath}
\usepackage{stmaryrd}
\usepackage{enumerate}
\usepackage{mathtools}

\usepackage[euler]{textgreek}
\usepackage{listings}
\usepackage{algorithm}
\usepackage{algpseudocode}

\newcommand{\vx}{\mathbf{x}}

\newcommand{\vn}{\mathbf{n}}

\newcommand{\vK}{\mathbf{K}}

\newcommand{\vI}{\mathbf{I}}

\newcommand{\cL}{\mathcal{L}}

\newcommand{\cT}{\mathcal{T}}

\newcommand{\cV}{\mathcal{V}}

\newcommand{\iO}{\int_{\Omega}}

\newcommand{\llb}{ \big\llbracket } 
\newcommand{\rrb}{ \big\rrbracket }

\newcommand{\bp}{\bar{p}}
\newcommand{\bz}{\bar{z}}

\newcommand{\bgm}{\bar{\gamma}}

\newcommand{\U}{\mathbb{U}}   
 
\newcommand{\R}{\mathbb{R}}   

\usepackage{amssymb, epsfig, enumerate, bm, xcolor}
\usepackage{amsmath,amsfonts}
\usepackage{mathrsfs}
\numberwithin{equation}{section}
\usepackage[top=1in, bottom=1in, left=1.3in, right=1.3in]{geometry}
\newtheorem{assumption}{Assumption}

\let\remark\remark         \let\endremark\endremark

\usepackage{epsfig, comment, graphicx, relsize}
\usepackage{subcaption}

\usepackage{tikz}
\usetikzlibrary{patterns}
\usepackage{hyperref}
\usepackage[many]{tcolorbox}
\usepackage{algorithmicx}
\usepackage{float}

\newcommand{\Eh}{{{\mathcal E}_h}} 
\newcommand{\Eho}{{{\mathcal E}^{I}_h}} 

\newcommand{\btau}{\boldsymbol{\tau}}

\theoremstyle{remark}
\newtheorem{remark}{Remark}[section]
\newtheorem{dfn}{Definition}[section]
\newtheorem{lem}{Lemma}[section]
\newtheorem{thm}{Theorem}[section]

\begin{document}

\title[DG Optimal Control]{A Monolithic Discontinuous Galerkin Framework for Darcy Optimal Control with Radon-Measure Tracking and Pointwise Control Constraints}      

\author{Seonghee Jeong}
\address{School of Computing and Data Science, Wentworth Institute of Technology,
Boston, MA 02115, U.S.A.}
\email{jeongs2@wit.edu}
\author{Sanghyun Lee}
\address{Department of Mathematics, Florida State University,
Tallahassee, FL 32306, U.S.A.}
\email{slee17@fsu.edu}



\begin{abstract}
{We study an elliptic optimal control problem governed by Darcy's equation in heterogeneous porous media, with pointwise box constraints on the control.
The objective functional is formulated via a Radon measure, which allows the desired pressure state to be tracked on observation sets of varying dimension, including points, curves, and subdomains, within a single formulation.
The state and adjoint equations are discretized by a symmetric interior penalty discontinuous Galerkin method, yielding a locally mass-conservative approximation that is robust across strong permeability discontinuities, while the control is approximated by piecewise constants.
The state, adjoint, and control are retained as primary unknowns in a single coupled optimality system and solved monolithically by a primal-dual active set strategy. 
We establish stability and well-posedness of the discretization and derive a priori $L^2$ error estimates for both variables. 
The principal difficulty is the reduced regularity of the adjoint state induced by the measure-valued tracking data
The analysis controls the resulting adjoint-control coupling through an intermediate adjoint driven by the continuous optimal state. 
Numerical experiments confirm the predicted convergence rates for point, curve, and subdomain observations, exhibit mesh-independent primal-dual active set iteration counts, and demonstrate local mass conservation.}
\end{abstract}

\maketitle

\section{Introduction}

In many subsurface flow applications, the computational task is not only to predict the pressure
field generated by prescribed sources, but also to determine admissible forcing mechanisms that
steer the system toward a desired pressure response. Such forcing mechanisms arise, for example,
through injection and extraction rates, distributed source or sink terms, or operational decisions
used in reservoir management, groundwater remediation, and geothermal energy
systems~\cite{culver1992dynamic,dixit2022stochastic,simon2015adjoint}. A central difficulty is that
the available pressure information is rarely distributed over the entire domain; instead,
measurements or design targets are typically concentrated at monitoring wells, along fracture-like
or geological features, or within selected regions of practical interest.
This motivates an optimal control formulation in which the state-tracking term acts only on
geometrically localized sets. Because the governing equation is a Darcy flow model in heterogeneous
porous media, the discretization should additionally preserve a local balance of mass.
These considerations lead naturally to a Radon-measure tracking functional combined with a locally
conservative discretization of the Darcy equation~\cite{brenner2024c0,jeong2025optimal}.

In this paper, we develop and analyze discontinuous Galerkin (DG) finite element methods for an
optimal control problem governed by Darcy's equation with pointwise box constraints on the control.
The tracking term is formulated by means of a Radon measure, allowing the desired pressure state to
be tracked on localized observation sets, including points, curves, and subdomains of the physical
domain. This formulation naturally covers monitoring-well observations, line-based measurements
(e.g\textcolor{red}{.}, fractures), and regional pressure targets within a unified framework.
The use of a symmetric interior penalty discontinuous Galerkin (SIPG) discretization further
provides local mass conservation, which is a key structural property for flow simulations in
heterogeneous porous media.
The control is approximated by piecewise constants, which is consistent with distributed source or
sink controls and leads to a natural finite-dimensional representation of the pointwise box
constraints.

Thus, for a bounded convex polygonal domain $\Omega\subset \R^2$ with boundary $\partial\Omega$, we
consider the following optimal control problem: find $(\bp,\bgm)$ such that
\begin{equation}\label{eqn:ocp}
(\bp,\bgm)
=
\operatorname*{arg\,min}_{(p,\gamma)\in \U}
\left\{
\frac{1}{2}\iO |p-p_d|^2 \, d\nu
+
\frac{\beta}{2}\iO |\gamma|^2 \, d\vx
\right\}
\end{equation}
subject to
\begin{subequations}\label{eqn:pde_constraints}
\begin{align}
-\nabla\cdot(\vK\nabla p)
&= f+\gamma
\quad \text{in } \Omega,\\
p
&= 0
\quad \text{on } \partial\Omega,
\end{align}
\end{subequations}
and the pointwise control constraints
\begin{equation}\label{eqn:control_constraints}
\gamma_-(\vx)
\leq
\gamma(\vx)
\leq
\gamma_+(\vx)
\quad \text{for a.e. } \vx\in\Omega.
\end{equation}
Here, $\U:=H^1_0(\Omega)\times L^2(\Omega)$ denotes the state-control space, and $\beta>0$.

The elliptic equation \eqref{eqn:pde_constraints}, which is also known as Darcy's flow equation,
models fluid flow in porous media. Here, $p:\Omega\rightarrow\mathbb{R}$ denotes the pressure,
$f\in L^2(\Omega)$ is a prescribed source or sink term, and $\vK$ is the isotropic permeability
tensor. For the theoretical analysis, we assume that $\vK\in (H^2(\Omega))^{2\times 2}$ is symmetric
and uniformly positive definite. That is, there exist positive constants $k_0$ and $k_1$ such that
\begin{equation}\label{eqn:uniform_ellipticity}
k_0\xi^T\xi
\leq
\xi^T\vK(\vx)\xi
\leq
k_1\xi^T\xi,
\qquad
\forall \xi\in\R^2,
\quad
\text{for a.e. } \vx\in\Omega.
\end{equation}
These assumptions are used to obtain the regularity estimates required in the a priori error
analysis. 
We emphasize that the a priori error estimates of
Section~\ref{sec:convergence_analysis} are established under this smoothness assumption on $\vK$, which
guarantees the elliptic regularity used throughout the analysis. The locally conservative,
coefficient-weighted SIPG discretization, however, is defined for merely piecewise-constant (and
hence discontinuous) permeability fields, and in Section~\ref{sec:numerics} we investigate
numerically its robustness for such heterogeneous media, for which the convergence theory is not
directly covered by these assumptions. 
In the numerical experiments, we also consider
heterogeneous permeability fields in order to illustrate the robustness of the locally conservative
DG discretization.

The pressure $p$ is the state variable, $\gamma$ is the control variable, and $p_d$ is the desired
state. Through the optimal control problem \eqref{eqn:ocp}, we seek a control $\gamma$ such that the
associated pressure $p$ approximates $p_d$ on the observation sets encoded by the measure $d\nu$,
while satisfying Darcy's equation \eqref{eqn:pde_constraints} and the pointwise control constraints
\eqref{eqn:control_constraints}. Moreover, we assume that $\gamma_\pm\in H^1(\Omega)$ and that
$\gamma_-(\vx)<\gamma_+(\vx) \text{ for a.e. } \vx\in\Omega.$

The main contributions of this paper are as follows. First, we consider a general Radon-measure
tracking functional, which unifies pointwise, curve-based, and subdomain-based observations of the
pressure state within a single formulation. Second, we employ a SIPG method with coefficient-weighted
numerical fluxes for the state and adjoint equations, thereby obtaining a locally mass-conservative
discretization that remains robust across strong permeability discontinuities and is suitable for
Darcy flow in heterogeneous porous media. Third, we retain the state, adjoint, and control variables
in a coupled optimality system and solve the constrained problem monolithically by a primal-dual
active set strategy, rather than eliminating the control or reconstructing it through a separate
post-processing step, so that the exact coupling between the discrete flow equation, adjoint
equation, and control constraints is preserved. The resulting analysis establishes stability and
well-posedness of the discrete optimality system and derives a priori $L^2$ error estimates for both
the state and control variables, with convergence rates that reflect the reduced regularity of the
adjoint state caused by the measure-data tracking term. While each of these ingredients has been
studied separately, to the best of our knowledge their combination---a locally conservative
discontinuous Galerkin discretization of a control-constrained Darcy optimal control problem with a
general Radon-measure tracking functional accommodating point, curve, and subdomain observations,
solved monolithically---has not been considered previously.

Optimal control problems governed by elliptic partial differential equations (PDEs) have been
extensively studied over the past decades; see, for example,
\cite{brenner2017new,casas2014new,falk1973approximation,hintermuller2007mesh,hinze2005variational,troltzsch2010finite,wan1992elliptic}.
DG methods and related finite element discretizations have also been developed for several classes of
PDE-constrained optimal control problems
\cite{allendes2022error,anh2020discontinuous,brenner2023symmetric,liu2024discontinuous,yucel2015discontinuous}.
In particular, \cite{brenner2023symmetric,liu2024discontinuous} study DG methods for elliptic optimal
control problems with pointwise state constraints. In these works, the analysis is primarily
formulated in terms of the state variable, and the control is recovered through the optimality system
or the PDE constraint. For control-constrained problems, DG and related nonconforming discretizations
have also been analyzed within general error-analysis frameworks \cite{chowdhury2015framework},
through DG discretizations for distributed control problems \cite{bommanaboyana2025dwdg}, and for
Dirichlet boundary control \cite{sau2025dirichlet}. These works are formulated with standard
Lebesgue-measure tracking terms, for which the adjoint equation does not exhibit the same
measure-induced loss of regularity considered here.

A separate line of work concerns tracking functionals with measure or point data, where the adjoint
equation carries a measure-valued right-hand side and the adjoint state has reduced regularity. Error
estimates for finite element approximations of such problems were initiated by
Casas~\cite{casas19852} and have been developed for pointwise-tracking control problems, for instance
in~\cite{allendes2022error}, where the control is discretized by piecewise constants or treated
variationally. Most closely related to the present setting, \cite{brenner2024c0} analyzes an interior
penalty method for an elliptic distributed control problem with general tracking and pointwise
\emph{state} constraints, and \cite{brenner2026point} studies point tracking with pointwise
\emph{control} constraints. These works are posed on the Poisson equation with conforming or $C^0$
interior penalty discretizations; they do not address the Darcy setting, local mass conservation, or
the coefficient-weighted DG discretization considered here, and the tracking data is restricted to
point observations rather than the unified point/curve/subdomain measure used in this paper.

The present paper differs from the works above in that the constraints are imposed directly on the
control variable, the tracking data is a general Radon measure, and the discretization is locally
conservative. We retain the state, adjoint, and control variables in the formulation and approximate
them simultaneously. The resulting discrete first-order optimality system is solved by a primal-dual
active set strategy, leading to a monolithic treatment of the constrained optimal control problem.
This approach is natural for control-constrained problems and allows us to
derive $L^2$-error estimates for both the state and control variables. Our analysis builds on
convergence techniques for elliptic optimal control problems with pointwise control constraints
developed in~\cite{brenner2024new,brenner2026point}, while incorporating the additional difficulties
caused by the coefficient-weighted DG discretization and the Radon-measure tracking term; in
particular, the coupling between the discrete adjoint and control errors is controlled through an
intermediate adjoint driven by the continuous optimal state. We note that our earlier
work~\cite{jeong2025optimal} considers Darcy optimal control with general tracking using a different
discretization framework based on a $C^0$ interior penalty method with the control eliminated from
the system, whereas the present paper develops a locally conservative symmetric interior penalty DG
discretization with pointwise control constraints, retains the control as a primary unknown, and
analyzes the associated measure-induced reduction of adjoint regularity.

The general tracking formulation broadens the applicability of the method to practical situations
where pressure observations or design targets are available only at monitoring wells (points), along
fractures (curves), or in selected subregions of the domain, and the combination of measure-valued
tracking, pointwise control constraints, and a locally conservative DG discretization is what
necessitates the careful error analysis developed in Section~\ref{sec:convergence_analysis}.

The remainder of this paper is organized as follows. Section~\ref{sec:wellposedness}
introduces the Radon-measure tracking functional, derives the first-order optimality conditions via
the transposition method, collects the relevant regularity results for the optimal variables, and
develops the SIPG discretization together with the discrete optimality system and local mass
conservation property. Section~\ref{sec:convergence_analysis} establishes $L^2$ error estimates for the state
and control variables, with convergence rates that reflect the reduced regularity of the adjoint
state caused by the measure-data tracking term. Section~\ref{sec:numerics} presents numerical
experiments on heterogeneous porous media confirming the predicted convergence rates for point,
curve, and subdomain observations and demonstrating the local mass conservation of the proposed
method. Section~\ref{sec:conclusions} collects concluding remarks.

\section{Well-posedness and Regularity Results}
\label{sec:wellposedness}

In this section, we show the well-posedness of the optimal control problem and recall the regularity results needed for the error analysis.
The elliptic distributed optimal control problem with pointwise control constraints has been widely studied \cite{falk1973approximation,troltzsch2010finite}.
In the present work, the cost functional contains a general tracking term expressed through a Radon measure, together with pointwise constraints imposed directly on the control variable.
Optimal control problems with general tracking and pointwise state constraints can be found, for example, in \cite{brenner2024c0}.
The main additional feature here is that the measure-data tracking term is combined with pointwise control constraints and a locally conservative discontinuous Galerkin discretization.

\subsection{Radon Measure}\label{sec:Radon}

The purpose of introducing the Radon measure in \eqref{eqn:ocp} is to allow the desired pressure state to be tracked on localized observation sets, including points, curves, and subdomains of $\Omega$.
Following \cite{brenner2024c0,jeong2025optimal}, we define the Radon measure $\nu$ by
\begin{equation}\label{radon}
    \int_\Omega q\,d\nu
    =
    \sum_{j=1}^J q(\mathscr{P}_j)w^j_\mathscr{P}
    +
    \sum_{l=1}^L\int_{\mathscr{C}_l}q w^l_\mathscr{C}\,ds
    +
    \sum_{m=1}^M \int_{\mathscr{E}_m} q w^m_\mathscr{E}\,d\vx,
\end{equation}
for any continuous function $q$ on $\bar{\Omega}$.
Here,
$\mathscr{P}=\{\mathscr{P}_1,\ldots,\mathscr{P}_J\}$ is a finite set of points with $\mathscr{P}_j\in\Omega$,
$\mathscr{C}=\{\mathscr{C}_1,\ldots,\mathscr{C}_L\}$ is a finite collection of sufficiently smooth curves with $\mathscr{C}_l\Subset\Omega$,
and
$\mathscr{E}=\{\mathscr{E}_1,\ldots,\mathscr{E}_M\}$ is a finite collection of measurable subdomains with $\mathscr{E}_m\Subset\Omega$.
The point weights $w^j_\mathscr{P}$ are nonnegative constants, while
the weight functions $w_\mathscr{C}^l$ and $w^m_\mathscr{E}$ are bounded nonnegative Borel measurable functions defined on the corresponding observation sets.
Thus $\nu$ is a finite nonnegative Radon measure on $\bar{\Omega}$.
The observation sets are allowed to overlap; in that case, the measure in \eqref{radon} is understood as the additive superposition of the point, curve, and subdomain contributions.

The desired state $p_d$ is prescribed on the support of $\nu$ by
\begin{equation}\label{eqn:desired_state_radon}
p_d :=
\begin{cases}
p_\mathscr{P} & \text{on } \mathscr{P},\\
p_\mathscr{C} & \text{on } \mathscr{C}\setminus\mathscr{P},\\
p_\mathscr{E} & \text{on } \mathscr{E}\setminus(\mathscr{C}\cup\mathscr{P}),
\end{cases}
\end{equation}
such that
\begin{equation}\label{eqn:pd_l2_nu}
\|p_d\|^2_{L^2(\Omega;\nu)}
:=
\int_\Omega |p_d|^2\,d\nu
<\infty.
\end{equation}
Here and below, $\mathscr{C}$ and $\mathscr{E}$ also denote the unions of the corresponding curve and subdomain observation sets.
The set-difference structure in \eqref{eqn:desired_state_radon} assigns a single value to $p_d$ on intersections of observation sets, using the priority order point, curve, and then subdomain. Since $\nu$ is supported on $\mathscr{P}\cup\mathscr{C}\cup\mathscr{E}$, the value of $p_d$ outside this set does not affect the tracking functional.
Equivalently, one may simply assume that $p_d\in L^2(\Omega;\nu)$. The explicit representation \eqref{eqn:desired_state_radon} is used only to emphasize the intended point, curve, and subdomain tracking cases.

\begin{remark}
For the state equation, we assume the following elliptic regularity estimate: if $g\in L^2(\Omega)$ and $p\in H^1_0(\Omega)$ satisfies
\begin{equation}\label{eqn:state_regulariy_assumption}
    -\nabla\cdot(\vK\nabla p)=g
    \quad \text{in } \Omega,
    \qquad
    p=0
    \quad \text{on } \partial\Omega,
\end{equation}
then
\begin{equation}\label{eqn:H2_regulariy_state}
    \|p\|_{H^2(\Omega)}
    \leq C\|g\|_{L^2(\Omega)}.
\end{equation}
This estimate holds, for example, under the convexity and coefficient-regularity assumptions stated in Section~1; see \cite{dauge1988elliptic,maz_i_a2010elliptic}.
Since $f+\gamma\in L^2(\Omega)$ for every admissible control $\gamma$, the associated state satisfies $p\in H^2(\Omega)$. 
Since $\Omega\subset\mathbb{R}^2$, the Sobolev embedding
\[
    H^2(\Omega)\hookrightarrow C(\bar{\Omega})
\]
implies that point evaluations of $p$ are well-defined.
Therefore, the tracking term in \eqref{eqn:ocp} is well-defined; see also \cite{adams2003sobolev,brenner2024c0}. We emphasize that this $H^2$ regularity pertains to the state $p$, whose source lies in $L^2(\Omega)$. The adjoint state, by contrast, is governed by a measure-valued source and therefore exhibits reduced regularity, as discussed in Section~\ref{sec:Regularities}.
\end{remark}

\subsection{First-Order Optimality Conditions}
\label{sec:FOC}

Here, we derive the first-order optimality conditions for the optimal control problem.
Because the tracking term is expressed through a Radon measure, the adjoint equation contains measure data. Therefore, instead of relying on a formal Lagrangian involving an $H^1_0(\Omega)$ adjoint, we derive the optimality system through the reduced control formulation and then introduce the adjoint state by transposition.

Let
\begin{equation}\label{eqn:Ugamma}
U_\gamma
:=
\left\{
\eta\in L^2(\Omega):
\gamma_-(\vx)\leq \eta(\vx)\leq \gamma_+(\vx)
\quad \text{for a.e. } \vx\in\Omega
\right\}.
\end{equation}
For each $\gamma\in U_\gamma$, let $S\gamma=p$ denote the unique solution of
\begin{equation}\label{eqn:state_control_to_state}
a^{\vK}(p,v)
=
\iO (f+\gamma)v\,d\vx
\qquad
\forall v\in H^1_0(\Omega),
\end{equation}
where
\begin{equation}\label{eqn:aK_def}
a^{\vK}(v,w)
:=
\iO \vK\nabla v\cdot\nabla w\,d\vx.
\end{equation}
Since $\vK$ is symmetric, the bilinear form $a^{\vK}(\cdot,\cdot)$ is symmetric whenever both sides are well-defined. In particular, this holds for $(v,w)\in W^{1,s}_0(\Omega)\times W^{1,r}_0(\Omega)$, where $W^{1,r}_0(\Omega):=\{v\in W^{1,r}(\Omega):v=0\text{ on }\partial\Omega\}$ with $s\in(1,2)$, $r=s/(s-1)$, and $1/s+1/r=1$.
The optimal control problem can then be written in reduced form as
\begin{equation}\label{eqn:reduced_problem}
\min_{\gamma\in U_\gamma}
j(\gamma)
:=
\frac{1}{2}\iO |S\gamma-p_d|^2\,d\nu
+
\frac{\beta}{2}\iO |\gamma|^2\,d\vx .
\end{equation}

Since $U_\gamma$ is a nonempty, closed, and convex subset of $L^2(\Omega)$, and since $\beta>0$, the reduced functional $j$ is strictly convex and coercive on $U_\gamma$. Hence the problem admits a unique optimal control $\bgm\in U_\gamma$, and the corresponding optimal state $\bp=S\bgm$ is unique. Moreover, the first-order optimality condition is necessary and sufficient:
\begin{equation}\label{eqn:reduced_first_order}
j'(\bgm)(\eta-\bgm)\geq 0
\qquad
\forall \eta\in U_\gamma .
\end{equation}

For a direction $\delta\gamma\in L^2(\Omega)$, let $\delta p=G\delta\gamma$ denote the solution of the linearized state equation
\begin{equation}\label{eqn:linearized_state}
a^{\vK}(\delta p,v)
=
\iO \delta\gamma\,v\,d\vx
\qquad
\forall v\in H^1_0(\Omega).
\end{equation}
Here $G:L^2(\Omega)\to H^1_0(\Omega)$ denotes the linear solution operator associated with the homogeneous control-to-state equation, that is, the linear part of the affine solution map $S$, so that $S\gamma=G\gamma+S(0)$. Under the assumptions of Section~\ref{sec:Radon}, elliptic regularity gives
\begin{equation}\label{eqn:linearized_state_regularity}
\delta p=G\delta\gamma \in H^2(\Omega)
\end{equation}
Hence $\delta p\in C(\bar\Omega)$ and point evaluations of $\delta p$ are well-defined. In particular, the tracking term in $j$ is G\^ateaux differentiable, and
\begin{equation}\label{eqn:reduced_derivative}
j'(\bgm)\delta\gamma
=
\iO (\bp-p_d)\,\delta p\,d\nu
+
\beta\iO \bgm\,\delta\gamma\,d\vx .
\end{equation}

We now introduce the adjoint state. Let
\begin{equation}\label{eqn:measure_mu_def}
\mu:=(\bp-p_d)\nu .
\end{equation}
This is a finite signed Radon measure because $\bp-p_d\in L^2(\Omega;\nu)$ and $\nu$ is finite. The adjoint state $\bz$ is first defined by transposition as the function satisfying
\begin{equation}
\label{eqn:adjoint_transposition_definition}
\iO \bz\,g\,d\vx
=
\langle \mu,Gg\rangle
=
\iO (\bp-p_d)\,Gg\,d\nu
\qquad
\forall g\in L^2(\Omega).
\end{equation}
The right-hand side is bounded on $L^2(\Omega)$ because $Gg\in H^2(\Omega)\hookrightarrow C(\bar\Omega)$ and
\begin{equation}\label{eqn:transposition_boundedness}
\|Gg\|_{C(\bar\Omega)}
\leq C\|Gg\|_{H^2(\Omega)}
\leq C\|g\|_{L^2(\Omega)} ,
\end{equation}
where $C$ depends only on $\Omega$ and $\vK$.
Therefore, by the Riesz representation theorem, $\bz\in L^2(\Omega)$ is well-defined.

By the measure-data elliptic regularity theory, see for example \cite{casas19852}, the adjoint satisfies the sharper regularity
\begin{equation}\label{eqn:adjoint_regular_first}
\bz\in W^{1,s}_0(\Omega)
\qquad
\forall s\in[1,2).
\end{equation}
The restriction $s<2$ reflects the possible presence of point masses in the tracking measure. Point observations are the most singular among the point, curve, and subdomain contributions in \eqref{radon}; if no point masses are present, higher adjoint regularity may be available, depending on the dimension and regularity of the observation support. After this regularity is established, the adjoint also satisfies the following weak formulation with nonsmooth right-hand side:
\begin{equation}\label{eqn:adjoint_transposition}
a^{\vK}(\bz,v)
=
\langle\mu,v\rangle
=
\iO (\bp-p_d)v\,d\nu
\qquad
\forall v\in W^{1,r}_0(\Omega), \quad r>2.
\end{equation}
This formulation is meaningful because $W^{1,r}_0(\Omega)\hookrightarrow C(\bar{\Omega})$ for $r>2$ in two space dimensions. 

Taking $g=\delta\gamma$ in \eqref{eqn:adjoint_transposition_definition}, with $\delta p=G\delta\gamma$, gives the transposition identity
\begin{equation}\label{eqn:transposition_identity}
\iO (\bp-p_d)\,\delta p\,d\nu
=
\iO \bz\,\delta\gamma\,d\vx
\qquad
\forall \delta\gamma\in L^2(\Omega).
\end{equation}

Combining \eqref{eqn:reduced_derivative} and \eqref{eqn:transposition_identity}, we obtain
\begin{equation}\label{eqn:reduced_derivative_adjoint}
j'(\bgm)(\eta-\bgm)
=
\iO (\beta\bgm+\bz)(\eta-\bgm)\,d\vx .
\end{equation}
Since we may choose $s\in(1,2)$ so that $\bz\in W^{1,s}_0(\Omega)\hookrightarrow L^2(\Omega)$ in two dimensions, the preceding expression is well-defined. Hence the control optimality condition becomes
\begin{equation}\label{eqn:control_vi}
\iO (\beta\bgm+\bz)(\eta-\bgm)\,d\vx \geq 0
\qquad
\forall \eta\in U_\gamma .
\end{equation}

Collecting the state equation, the adjoint equation, and the control variational inequality, the continuous first-order optimality system reads: find $(\bp,\bgm,\bz)$ such that
\begin{subequations}\label{optimality_conditions}
\begin{align}
a^{\vK}(\bp,v)
&=
\iO (f+\bgm)v\,d\vx
&&
\forall v\in H^1_0(\Omega), \label{eqn:darcy}\\
a^{\vK}(\bz,v)
&=
\iO (\bp-p_d)v\,d\nu
&&
\forall v\in W^{1,r}_0(\Omega), \quad r>2, \label{eqn:adjoint_state}\\
(\beta\bgm+\bz,\eta-\bgm)_{L^2(\Omega)}
&\geq 0
&&
\forall \eta\in U_\gamma . \label{stationary_optimality_control_ineq}
\end{align}
\end{subequations}
The state equation is posed in the standard weak sense, while the adjoint equation is posed in the transposition sense because of the measure-valued right-hand side.

Equivalently, the optimal pair $(\bp,\bgm)$ is characterized by the reduced variational inequality
\begin{equation}\label{eqn:consolidated_vi}
\iO(\bp-p_d)(p-\bp)\,d\nu
+
\beta\iO\bgm(\gamma-\bgm)\,d\vx
\geq 0
\qquad
\forall \gamma\in U_\gamma,
\end{equation}
where $p$ and $\bp$ are the states associated with the controls $\gamma$ and $\bgm$, respectively, through \eqref{eqn:state_control_to_state}.

For later reference, we fix an arbitrary $s\in(1,2)$ and introduce the product constraint set
\begin{equation}\label{eqn:continuous_admissible_product_set}
\mathcal{U}_{ad}
:=
\left\{
(p,\gamma,z)\in H^1_0(\Omega)\times L^2(\Omega)\times W^{1,s}_0(\Omega):
\gamma_-(\vx)\leq \gamma(\vx)\leq \gamma_+(\vx)
\ \text{for a.e. } \vx\in\Omega
\right\}.
\end{equation}
Here the homogeneous Dirichlet conditions on the state and adjoint are encoded in $H^1_0(\Omega)$ and $W^{1,s}_0(\Omega)$, respectively. The actual optimality system additionally imposes the state equation, the measure-data adjoint equation, and the control variational inequality.

\subsection{Regularities}\label{sec:Regularities}
In this subsection, we collect the regularity properties of the optimal variables that will be used in the error analysis. The existence and uniqueness of the optimal control have already been obtained from the reduced formulation in Section~\ref{sec:FOC}. We therefore focus on the consequences of the measure-data adjoint equation and the pointwise projection formula for the control.

Let $(\bp,\bgm,\bz)$ satisfy the first-order optimality system \eqref{optimality_conditions}. Then $\bgm\in U_\gamma$, $\bp$ is the associated state, and $\bz$ is the adjoint state defined by transposition. Conversely, if $\bgm\in U_\gamma$, together with its associated state $\bp$ and adjoint state $\bz$, satisfies the variational inequality \eqref{stationary_optimality_control_ineq}, then $\bgm$ is the unique optimal control \cite{troltzsch2010optimal}.

\begin{remark}
A necessary and sufficient condition for the variational inequality \eqref{stationary_optimality_control_ineq} is that, for almost every $\vx\in\Omega$ \cite{troltzsch2010optimal},
\begin{equation}\label{control_vi}
\bgm(\vx) =
\begin{cases}
\gamma_-(\vx) & \text{if } \bz(\vx)+\beta\bgm(\vx) > 0,\\
\in[\gamma_-(\vx),\gamma_+(\vx)] & \text{if } \bz(\vx)+\beta\bgm(\vx) = 0,\\
\gamma_+(\vx) & \text{if } \bz(\vx)+\beta\bgm(\vx) < 0.
\end{cases}
\end{equation}
\end{remark}

\begin{dfn}
For real numbers $a\leq b$, the projection operator $\mathbb{P}_{[a,b]}$ is defined by
\begin{equation}\label{eqn:projection_operator}
\mathbb{P}_{[a,b]}(v):=\min\{b,\max\{a,v\}\}.
\end{equation}
\end{dfn}

The optimal control $\bgm$ satisfies, together with the adjoint state $\bz$, the projection formula
\begin{equation}\label{adjoint_state}
\bgm(\vx)
=
\mathbb{P}_{[\gamma_-(\vx),\gamma_+(\vx)]}
\left(-\frac{1}{\beta}\bz(\vx)\right)
\qquad
\text{for a.e. } \vx\in\Omega.
\end{equation}
Since $\gamma_\pm\in H^1(\Omega)\subset W^{1,s}(\Omega)$ for all $s\in[1,2]$ and $\bz\in W^{1,s}_0(\Omega)$ for all $s\in[1,2)$, the Lipschitz continuity of the pointwise projection implies
\begin{equation}\label{control_regul}
\bgm\in W^{1,s}(\Omega)
\qquad
\forall\,s\in[1,2).
\end{equation}

For the state equation, since $f\in L^2(\Omega)$ and $\bgm\in L^2(\Omega)$, the source term $f+\bgm$ belongs to $L^2(\Omega)$. Therefore, under the elliptic regularity assumption stated in Section~\ref{sec:Radon},
\begin{equation}\label{state_H2_regul}
\bp\in H^2(\Omega),
\qquad
\|\bp\|_{H^2(\Omega)}
\leq
C\left(\|f\|_{L^2(\Omega)}+\|\bgm\|_{L^2(\Omega)}\right).
\end{equation}

If, in addition, $f\in W^{1,s}(\Omega)$ for some fixed $s\in[1,2)$, then $f+\bgm\in W^{1,s}(\Omega)$ and interior elliptic regularity yields
\begin{equation}\label{state_regul_loc}
\bp\in W^{3,s}_{\mathrm{loc}}(\Omega).
\end{equation}

\subsection{The Lagrange Multiplier}
\label{sec:Lagrange_multiplier}

We now record two properties relating the optimal control and the adjoint state that will be used in the error analysis.
Following \cite{brenner2024new}, we define
\begin{equation}\label{lambda_L}
    \lambda := \beta\bgm + \bz .
\end{equation}
The function $\lambda$ can be interpreted as the Lagrange multiplier associated with the box constraints in \eqref{eqn:control_constraints}. By \eqref{eqn:adjoint_regular_first} and  \eqref{control_regul},
\begin{equation}\label{eqn:lambda_regular}
    \lambda\in W^{1,s}(\Omega)
    \qquad
    \forall s\in[1,2).
\end{equation}

Let
\begin{equation}\label{L_decomposition}
    \lambda = \lambda^- + \lambda^+,
\end{equation}
where
\begin{equation}\label{L_minmax}
    \lambda^- := \min\{\lambda,0\}\leq 0
    \qquad\text{and}\qquad
    \lambda^+ := \max\{\lambda,0\}\geq 0 .
\end{equation}
Since truncation at zero is a Lipschitz operation,
\begin{equation}\label{eqn:lambda_pm_regular}
    \lambda^+,\lambda^-\in W^{1,s}(\Omega)
    \qquad
    \forall s\in[1,2).
\end{equation}

By the projection formula \eqref{adjoint_state}, the optimal control satisfies
\begin{equation}\label{control_vi_2}
\bgm(\vx)
=
\begin{cases}
\gamma_-(\vx) 
& \text{if } -\bz(\vx)/\beta\leq \gamma_-(\vx), \\[2mm]
-\bz(\vx)/\beta 
& \text{if } \gamma_-(\vx)<-\bz(\vx)/\beta<\gamma_+(\vx), \\[2mm]
\gamma_+(\vx) 
& \text{if } -\bz(\vx)/\beta\geq \gamma_+(\vx),
\end{cases}
\qquad
\text{for a.e. } \vx\in\Omega .
\end{equation}
Equivalently, in terms of the multiplier $\lambda=\beta\bgm+\bz$, the following pointwise complementarity relations hold:
\begin{equation}\label{eqn:pointwise_complementarity}
    \lambda^+(\bgm-\gamma_-)=0,
    \qquad
    \lambda^-(\bgm-\gamma_+)=0
    \quad
    \text{a.e. in } \Omega .
\end{equation}
Consequently,
\begin{equation}\label{L_complementary}
    \iO \lambda^+(\bgm-\gamma_-)\,d\vx
    =
    0
    =
    \iO \lambda^-(\bgm-\gamma_+)\,d\vx .
\end{equation}




\subsection{Discontinuous Galerkin Finite Element Methods}

Let $\cT_h=\{T\}$ be a shape-regular partition of the domain $\Omega$ into quadrilateral elements with mesh size $h=\max_{T\in\cT_h}h_T$, where $h_T$ is the length of the diagonal of $T$.

We denote by $\Eh$ the set of all edges in the mesh, by $\Eho$ the set of all interior edges, and by $\cV_h$ the set of vertices of $\cT_h$.
For each $T\in\cT_h$, we denote the boundary of $T$ by $\partial T$.
For each interior edge $e\in\Eho$ shared by two neighboring elements $T_+$ and $T_-$, we fix the unit normal vector $\vn_e$ pointing from $T_+$ to $T_-$.
For a boundary edge $e\in\Eh\setminus\Eho$, $\vn_e$ denotes the unit outward normal of $\Omega$.

In this section, we assume that the permeability is represented by a piecewise constant scalar coefficient $\kappa$, with
$\kappa|_T=\kappa_T>0, 
\forall T\in\cT_h.$
Equivalently, the permeability tensor is $\vK=\kappa\vI$.
For an interior edge $e=\partial T_+\cap\partial T_-$, we denote
$\kappa^+ := \kappa|_{T_+},
\kappa^- := \kappa|_{T_-}.$
For a boundary edge $e\subset\partial T\cap\partial\Omega$, we set $\kappa_e:=\kappa|_T$.

We define the broken Sobolev space, for $s>3/2$, by
\begin{equation}\label{eqn:broken_sobolev_space}
    H^s(\Omega;\cT_h)
    :=
    \{v\in L^2(\Omega): v|_T\in H^s(T)\ \ \forall\,T\in\cT_h\}.
\end{equation}
For $v\in H^1(\Omega;\cT_h)$ and an interior edge $e\in\Eho$ shared by $T_\pm$, we define
\begin{equation}\label{eqn:trace_definition}
    v^\pm := (v|_{T_\pm})|_e .
\end{equation}
The scalar jump is defined by
\begin{equation}\label{eqn:jump_definition}
    \llb v \rrb := v^+ - v^- ,
\end{equation}
where the orientation is determined by the fixed normal vector $\vn_e$ from $T_+$ to $T_-$.
For a boundary edge $e\in\Eh\setminus\Eho$, which belongs to a single element $T$, we set
\begin{equation}\label{eqn:boundary_jump_definition}
    \llb v \rrb := v|_T .
\end{equation}

For interior edges, we use coefficient-dependent weighted averages. Define
\begin{equation}\label{eqn:beta_e_definition}
    \beta_e
    :=
    \frac{\kappa^-}{\kappa^+ + \kappa^-},
    \qquad
    1-\beta_e
    =
    \frac{\kappa^+}{\kappa^+ + \kappa^-}.
\end{equation}
For a scalar quantity $\zeta$ and a vector quantity $\btau$, define the weighted averages on $e\in\Eho$ by
\begin{equation}\label{eqn:weighted_average_definition}
    \{\!\{\zeta\}\!\}_{\beta_e}
    :=
    \beta_e\zeta^+ + (1-\beta_e)\zeta^-,
    \qquad
    \{\!\{\btau\}\!\}_{\beta_e}
    :=
    \beta_e\btau^+ + (1-\beta_e)\btau^- .
\end{equation}
The usual arithmetic average is denoted by
\begin{equation}\label{eqn:arithmetic_average_definition}
    \{\!\{\zeta\}\!\}
    :=
    \frac12(\zeta^+ + \zeta^-),
    \qquad
    \{\!\{\btau\}\!\}
    :=
    \frac12(\btau^+ + \btau^-) .
\end{equation}
On a boundary edge $e\in\Eh\setminus\Eho$, we set
\begin{equation}\label{eqn:boundary_average_definition}
    \{\!\{\zeta\}\!\}_{\beta_e}
    =
    \{\!\{\zeta\}\!\}
    :=
    \zeta,
    \qquad
    \{\!\{\btau\}\!\}_{\beta_e}
    =
    \{\!\{\btau\}\!\}
    :=
    \btau .
\end{equation}

The harmonic edge coefficient is defined by
\begin{equation}\label{eqn:kappa_e_definition}
\kappa_e
:=
\begin{cases}
\dfrac{2\kappa^+\kappa^-}{\kappa^+ + \kappa^-},
& e=\partial T_+\cap\partial T_- \in \Eho,\\[3mm]
\kappa_h|_T,
& e\subset \partial T\cap\partial\Omega .
\end{cases}
\end{equation}
With the choice \eqref{eqn:beta_e_definition}, the weighted average of the diffusive flux satisfies, on each interior edge,
\begin{equation}\label{eqn:weighted_flux_identity}
\{\!\{\kappa_h\nabla v\}\!\}_{\beta_e}
=
\kappa_e \{\!\{\nabla v\}\!\}.
\end{equation}
Indeed,
\[
\{\!\{\kappa_h\nabla v\}\!\}_{\beta_e}
=
\frac{\kappa^-}{\kappa^+ + \kappa^-}\kappa^+(\nabla v)^+
+
\frac{\kappa^+}{\kappa^+ + \kappa^-}\kappa^-(\nabla v)^-
=
\frac{\kappa^+\kappa^-}{\kappa^+ + \kappa^-}
\bigl((\nabla v)^+ + (\nabla v)^-\bigr)
=
\kappa_e\{\!\{\nabla v\}\!\}.
\]

We define the discontinuous finite element spaces $V_h$ and $W_h$ by
\begin{align}
    V_h
    &:=
    \{v\in L^2(\Omega): v|_T\in \mathbb{Q}_1(T)\ \ \forall\, T\in\cT_h\},
    \label{eqn:Vh}\\
    W_h
    &:=
    \{v\in L^2(\Omega): v|_T\in \mathbb{Q}_0(T)\ \ \forall\, T\in\cT_h\}.
    \label{eqn:Wh}
\end{align}
Here $\mathbb{Q}_k(T)$ denotes the space of polynomials of degree at most $k$ on $T$.
The homogeneous Dirichlet boundary condition is imposed weakly through the boundary terms of the symmetric interior penalty bilinear form below.

Using the SIPG discretization with coefficient-weighted flux averages, we define
\begin{align}\label{DG_bilinear}
    a_h^{\kappa}(v,w)
    &:=
    \sum_{T\in\cT_h}\int_T \kappa_h\nabla v\cdot\nabla w\,d\vx
    -
    \sum_{e\in\Eh}\int_e
    \{\!\{\kappa_h\nabla v\}\!\}_{\beta_e}\cdot\vn_e\,
    \llb w\rrb\,ds
    \nonumber\\
    &\quad
    -
    \sum_{e\in\Eh}\int_e
    \{\!\{\kappa_h\nabla w\}\!\}_{\beta_e}\cdot\vn_e\,
    \llb v\rrb\,ds
    +
    \sum_{e\in\Eh}\int_e
    \sigma \kappa_e h_e^{-1}\,
    \llb v\rrb\,\llb w\rrb\,ds .
\end{align}
Here $\sigma>0$ is a dimensionless penalty parameter chosen sufficiently large, and $h_e$ denotes the length of the edge $e$.
For sufficiently large $\sigma$, the bilinear form $a_h^{\kappa}(\cdot,\cdot)$ is coercive with respect to the DG energy norm introduced below.

\subsection{Discrete Problem}
\label{sec:discrete_problem}

Let $\mathcal{L}_h:H^s(\Omega;\cT_h)\to V_h$ be defined by
\begin{equation}\label{discrete_L}
    (\mathcal{L}_h v,w_h)_{L^2(\Omega)}
    =
    a_h^\kappa(v,w_h)
    \qquad
    \forall\,w_h\in V_h .
\end{equation}
Here $s>3/2$ is chosen so that the traces appearing in $a_h^\kappa(\cdot,\cdot)$ are well-defined.

Since $V_h$ is discontinuous, pairing the tracking measure against $v_h\in V_h$ requires a convention on lower-dimensional observation sets. We impose the following mesh-compatibility assumption.

\begin{assumption}\label{ass:observation_mesh}
The meshes $\{\cT_h\}$ are compatible with the observation sets in the following sense: each point $\mathscr{P}_j$ lies in the interior of a mesh element, and, for every curve $\mathscr{C}_l$, the intersection of $\mathscr{C}_l$ with the mesh skeleton has one-dimensional measure zero.
\end{assumption}

Under Assumption~\ref{ass:observation_mesh}, point evaluations of $v_h\in V_h$ at each $\mathscr{P}_j$ are single-valued. Moreover, the trace of $v_h$ is single-valued $ds$-a.e.\ on each curve $\mathscr{C}_l$. Hence the measure pairing $(\,\cdot\,,v_h)_{L^2(\Omega;\nu)}$ is well-defined on $V_h$.

We define the discrete admissible control set by
\begin{equation}\label{eqn:discrete_control_set}
U_{\gamma,h}
:=
\left\{
\eta_h\in W_h:
Q_h\gamma_-|_T
\leq
\eta_h|_T
\leq
Q_h\gamma_+|_T
\quad
\forall\,T\in\cT_h
\right\},
\end{equation}
where $Q_h:L^2(\Omega)\to W_h$ is the elementwise $L^2$ projection operator. We assume that
\begin{equation}\label{Q_h_positive}
    Q_h\eta|_T \geq 0
    \qquad
    \text{if } \eta\geq0 \text{ a.e. in } T .
\end{equation}
This positivity property implies that the discrete bounds preserve the ordering $Q_h\gamma_-\leq Q_h\gamma_+$ elementwise.

For every $\epsilon>0$ there exists $s_\epsilon\in(1,2)$ sufficiently close to $2$ such that

\begin{equation}\label{eqn:control_projection_estimate}
\|\bgm-Q_h\bgm\|_{L^2(\Omega)}
\leq
C h^{1-\epsilon}\|\bgm\|_{W^{1,s_\epsilon}(\Omega)},
\end{equation}
This estimate is one of the places where the reduced adjoint regularity affects the control error analysis.

For each $\gamma_h\in U_{\gamma,h}$, let $p_h\in V_h$ denote the unique solution of
\begin{equation}\label{eqn:discrete_state_map}
a_h^\kappa(p_h,v_h)
=
\iO (f+\gamma_h)v_h\,d\vx
\qquad
\forall\,v_h\in V_h .
\end{equation}
The homogeneous Dirichlet condition is enforced weakly through the boundary terms of $a_h^\kappa(\cdot,\cdot)$ and is not imposed strongly in $V_h$.

The discrete problem can be written in constrained form: find $(\bp_h,\bgm_h)\in V_h\times U_{\gamma,h}$ minimizing
\begin{equation}\label{eqn:discrete_ocp}
    \frac1{2}\iO |p_h-p_d|^2\,d\nu
    +
    \frac{\beta}{2}\iO|\gamma_h|^2\,d\vx
\end{equation}
subject to
\begin{equation}\label{discrete_pde_constraint}
a_h^\kappa(p_h,v_h)
=
\iO (f+\gamma_h)v_h\,d\vx
\qquad
\forall\,v_h\in V_h .
\end{equation}

Note that $W_h$ is a closed subspace of $L^2(\Omega)$, whereas $V_h$ is not a subspace of $H^1_0(\Omega)$.
Since $U_{\gamma,h}$ is a nonempty, closed, and convex subset of the finite-dimensional space $W_h$, and since $\beta>0$, the discrete problem \eqref{eqn:discrete_ocp}-\eqref{discrete_pde_constraint} has a unique solution.

As in the continuous problem, the discrete optimal pair $(\bp_h,\bgm_h)$ is characterized by
\begin{equation}\label{discrete_characterization}
    \iO(\bp_h-p_d)(p_h-\bp_h)\,d\nu
    +
    \beta\iO \bgm_h(\gamma_h-\bgm_h)\,d\vx
    \geq0
    \qquad
    \forall\,\gamma_h\in U_{\gamma,h}.
\end{equation}

Let the discrete adjoint state $\bz_h\in V_h$ be defined by
\begin{equation}\label{discrete_adjoint}
    a_h^\kappa(\bz_h,v_h)
    =
    (\bp_h-p_d,v_h)_{L^2(\Omega;\nu)}
    \qquad
    \forall\,v_h\in V_h .
\end{equation}
The discrete first-order optimality conditions are
\begin{subequations}\label{discrete_optimality_conditions}
\begin{align}
a_h^\kappa(\bp_h,v_h)
&=
(f+\bgm_h,v_h)_{L^2(\Omega)}
&&
\forall v_h\in V_h,
\label{discrete_optimality_condition1}\\
a_h^\kappa(\bz_h,v_h)
&=
(\bp_h-p_d,v_h)_{L^2(\Omega;\nu)}
&&
\forall v_h\in V_h,
\label{discrete_optimality_condition2}\\
(\beta\bgm_h+\bz_h,\eta_h-\bgm_h)_{L^2(\Omega)}
&\geq 0
&&
\forall \eta_h\in U_{\gamma,h}.
\label{discrete_optimality_condition3}
\end{align}
\end{subequations}

For later reference, we introduce the discrete product constraint set
\begin{equation}\label{discrete_admissible_set}
    \mathcal{U}_{ad,h}
    :=
    \left\{
    (p_h,\gamma_h,z_h)\in V_h\times W_h\times V_h:
    Q_h\gamma_-|_T
    \leq
    \gamma_h|_T
    \leq
    Q_h\gamma_+|_T
    \quad
    \forall T\in\cT_h
    \right\}.
\end{equation}
As in the continuous case, this product set records the discrete control bounds but does not by itself impose the discrete state equation, the adjoint equation, or the variational inequality.

Define the mesh-dependent DG energy norm $|\cdot|_{H^1(\Omega;\cT_h)}$ by
\begin{equation}\label{brokenH1norm}
    |v|^2_{H^1(\Omega;\cT_h)}
    :=
    \sum_{T\in \cT_h}|v|^2_{H^1(T)}
    +
    \sum_{e\in\Eh}
    \frac{\sigma \kappa_e}{h_e}
    \bigl\|\llb v\rrb\bigr\|^2_{L^2(e)} ,
\end{equation}
with the edge weight $\kappa_e$ from \eqref{eqn:kappa_e_definition}, consistent with the penalty term in \eqref{DG_bilinear}.
Because the boundary-edge penalty terms are included, this quantity is a norm on $V_h$.

To establish well-posedness of the DG discretization, we verify the continuity and coercivity of $a_h^\kappa(\cdot,\cdot)$.
By the trace theorem, the standard inverse estimates, the definition of $\kappa_e$, and the uniform ellipticity \eqref{eqn:uniform_ellipticity}, there is a constant $C>0$, independent of $h$, such that for every $v_h\in V_h$,
\begin{equation}\label{est:inverse}
    \sum_{e\in\Eh}\kappa_e^{-1}h_e
    \|\{\vK\nabla v_h\}\|^2_{L^2(e)}
    \leq
    C\sum_{T\in\cT_h}|v_h|^2_{H^1(T)} .
\end{equation}
From the Cauchy--Schwarz inequality and \eqref{est:inverse}, we obtain
\begin{align}\label{est:sym}
    \sum_{e\in\Eh}\int_e\{\vK\nabla v_h\}\llb w_h\rrb\,ds
    &\leq
    \left(
    \sum_{e\in\Eh}
    \kappa_e^{-1}h_e
    \|\{\vK\nabla v_h\}\|_{L^2(e)}^2
    \right)^{1/2}
    \left(
    \sum_{e\in\Eh}
    \kappa_e h_e^{-1}
    \|\llb w_h\rrb\|_{L^2(e)}^2
    \right)^{1/2}
    \nonumber\\
    &\leq
    C
    \left(
    \sum_{T\in\cT_h}|v_h|_{H^1(T)}^2
    \right)^{1/2}
    \left(
    \sum_{e\in\Eh}
    \kappa_e h_e^{-1}
    \|\llb w_h\rrb\|_{L^2(e)}^2
    \right)^{1/2}.
\end{align}
Hence, by \eqref{brokenH1norm} and \eqref{est:sym}, there exists a constant $C_\dagger>0$ such that
\begin{equation}\label{continuity_DG}
    |a_h^\kappa(v_h,w_h)|
    \leq
    C_\dagger
    |v_h|_{H^1(\Omega;\cT_h)}
    |w_h|_{H^1(\Omega;\cT_h)}
    \qquad
    \forall v_h,w_h\in V_h .
\end{equation}
Also, using \eqref{eqn:uniform_ellipticity}, \eqref{est:inverse}, \eqref{est:sym}, and the arithmetic--geometric mean inequality, there exists a constant $C_\ddagger>0$ such that
\begin{align}\label{coercivity_DG}
    a_h^\kappa(v_h,v_h)
    &\geq
    k_0\sum_{T\in\cT_h}|v_h|^2_{H^1(T)}
    -
    C
    \left(
    \sum_{T\in\cT_h}|v_h|^2_{H^1(T)}
    \right)^{1/2}
    \left(
    \sum_{e\in \Eh}
    \kappa_e h_e^{-1}
    \|\llb v_h\rrb\|_{L^2(e)}^2
    \right)^{1/2}
    \nonumber\\
    &\quad
    +
    \sigma\sum_{e\in \Eh}
    \kappa_e h_e^{-1}
    \|\llb v_h\rrb\|_{L^2(e)}^2
    \nonumber\\
    &\geq
    \frac{k_0}{2}
    \sum_{T\in\cT_h}|v_h|^2_{H^1(T)}
    +
    \left(\sigma-\frac{C^2}{2k_0}\right)
    \sum_{e\in \Eh}
    \kappa_e h_e^{-1}
    \|\llb v_h\rrb\|_{L^2(e)}^2
    \nonumber\\
    &\geq
    C_\ddagger
    |v_h|^2_{H^1(\Omega;\cT_h)}
    \qquad
    \forall v_h\in V_h,
\end{align}
provided $\sigma$ is sufficiently large, which is assumed henceforth.



\subsection{Local Mass Conservation}
\label{sec:local_mass_conservation}

One of the key properties of DG methods is local mass conservation.
Let $T\in\cT_h$ and let $\vn_T$ denote the outward unit normal of $T$ on $\partial T$.
For each edge $e\subset\partial T$, define
\begin{equation}\label{eqn:orientation_sign}
s_{T,e}
:=
\begin{cases}
1, & \text{if } e\in\Eho \text{ and } T=T_+,\\
-1, & \text{if } e\in\Eho \text{ and } T=T_-,\\
1, & \text{if } e\subset\partial\Omega .
\end{cases}
\end{equation}
Then $\vn_T=s_{T,e}\vn_e$ on each edge $e\subset\partial T$.

For the continuous state equation, integration by parts on $T$ gives
\begin{equation}\label{lmc1}
    \int_T -\nabla\cdot(\vK\nabla p)\,d\vx
    =
    -\int_{\partial T}\vK\nabla p\cdot \vn_T\,ds
    =
    \int_T F\,d\vx,
    \qquad
    F:=f+\gamma .
\end{equation}

We now derive the discrete analogue. Let $\chi_T\in V_h$ be the elementwise constant test function defined by
\[
\chi_T|_T=1,
\qquad
\chi_T|_{T'}=0
\quad
\text{for } T'\neq T .
\]
Since $\nabla\chi_T=0$ on every element, the element-gradient term and the symmetric consistency term involving $\{\vK\nabla \chi_T\}$ vanish in \eqref{DG_bilinear}. Moreover, for each $e\subset\partial T$,
\[
\llb \chi_T\rrb = s_{T,e}.
\]
Taking $v_h=\chi_T$ in the discrete state equation \eqref{discrete_pde_constraint}, with
\[
F_h:=f+\bgm_h,
\]
therefore yields
\begin{equation}\label{eqn:discrete_local_balance_raw}
    \sum_{e\subset\partial T}
    \int_e
    s_{T,e}
    \left(
    -\{\vK\nabla \bp_h\}
    +
    \sigma\kappa_e h_e^{-1}\llb \bp_h\rrb
    \right)\,ds
    =
    \int_T F_h\,d\vx .
\end{equation}

We define the outward numerical normal flux on $e\subset\partial T$ by
\begin{equation}\label{flux}
    \widehat q_{h,T}
    :=
    s_{T,e}
    \left(
    -\{\vK\nabla \bp_h\}
    +
    \sigma\kappa_e h_e^{-1}\llb \bp_h\rrb
    \right).
\end{equation}
Here $\{\vK\nabla \bp_h\}$ is the averaged normal flux with respect to the fixed normal $\vn_e$ introduced in Section~\ref{sec:discrete_problem}, and the factor $s_{T,e}$ converts it into the outward flux for the element $T$.

Then \eqref{eqn:discrete_local_balance_raw} becomes the local conservation law
\begin{equation}\label{eqn:discrete_local_mass_conservation}
    \sum_{e\subset\partial T}
    \int_e \widehat q_{h,T}\,ds
    =
    \int_T (f+\bgm_h)\,d\vx
    \qquad
    \forall T\in\cT_h .
\end{equation}
Thus the numerical flux \eqref{flux} balances the element source on every element, including boundary elements, where the boundary jump and flux conventions of the SIPG formulation apply.

\section{Convergence Analysis}
\label{sec:convergence_analysis}

In this section, we derive a priori error estimates for the discrete optimal control problem. The analysis combines three ingredients: the stability and approximation properties of the SIPG discretization, the reduced regularity of the adjoint state induced by the measure-valued tracking term, and the variational inequality arising from the pointwise control constraints. 
We first collect several auxiliary results for the DG discretization, then introduce the interpolation, and Ritz
operator.
Throughout, $C$ denotes a generic positive constant, independent of the mesh size $h$, that may change from line to line.

We write
\begin{equation}\label{eqn:continuous_operator_L}
\cL v := -\nabla\cdot(\vK\nabla v).
\end{equation}
In addition to the DG energy norm $|\cdot|_{H^1(\Omega;\cT_h)}$ defined in \eqref{brokenH1norm}, we use the full broken DG norm
\begin{equation}\label{eqn:full_broken_DG_norm}
\|v\|^2_{H^1(\Omega;\cT_h)}
:=
\|v\|^2_{L^2(\Omega)}
+
|v|^2_{H^1(\Omega;\cT_h)} .
\end{equation}
Since the boundary-edge penalty terms are included in \eqref{brokenH1norm}, the discrete Poincar\'e inequality gives
\begin{equation}\label{eqn:DG_poincare}
\|v_h\|_{L^2(\Omega)}
\leq
C |v_h|_{H^1(\Omega;\cT_h)}
\qquad
\forall v_h\in V_h .
\end{equation}

We begin with two stability results for the discrete bilinear form and a regularity-dependent bound used in the approximation analysis. The first lemma shows that the discrete solution operator is bounded in both the $L^2$ and DG energy norms; it underlies the recovery of the state error from the control error in Section~\ref{sec:error_estimates}.

\begin{lem}\label{lem:discrete_bound}
Let $g\in L^2(\Omega)$ and let $v_h\in V_h$ satisfy
\begin{equation}\label{useful}
    a_h^\kappa(v_h,u_h)
    =
    \iO g\,u_h\,d\vx
    \qquad
    \forall u_h\in V_h.
\end{equation}
Then
\begin{align}
    \|v_h\|_{L^2(\Omega)}
    &\leq
    C\|g\|_{L^2(\Omega)}, \label{eqn:discrete_bound_L2}\\
    |v_h|_{H^1(\Omega;\cT_h)}
    &\leq
    C\|g\|_{L^2(\Omega)}. \label{eqn:discrete_bound_H1}
\end{align}
\end{lem}

\begin{proof}
Taking $u_h=v_h$ in \eqref{useful} and using the coercivity estimate \eqref{coercivity_DG}, the Cauchy--Schwarz inequality, and the discrete Poincar\'e inequality \eqref{eqn:DG_poincare}, we obtain
\begin{align*}
    |v_h|^2_{H^1(\Omega;\cT_h)}
    &\leq
    C a_h^\kappa(v_h,v_h)
    =
    C\iO g\,v_h\,d\vx  \\
    &\leq
    C\|g\|_{L^2(\Omega)}\|v_h\|_{L^2(\Omega)}
    \leq
    C\|g\|_{L^2(\Omega)}|v_h|_{H^1(\Omega;\cT_h)} .
\end{align*}
This proves \eqref{eqn:discrete_bound_H1}. The $L^2$ estimate \eqref{eqn:discrete_bound_L2} then follows immediately from \eqref{eqn:DG_poincare}.
\end{proof}

The next result is a discrete Sobolev inequality, which is needed to control the pointwise part of the measure-tracking term. 

\begin{lem}\label{lem:discreteSobolevIneq}
(Discrete Sobolev Inequality \cite{brenner2004discrete})
For $v_h\in V_h$,
\begin{equation}\label{eqn:discrete_sobolev}
    \|v_h\|_{L^\infty(\Omega)}
    \leq
    C(1+|\ln h|)^{1/2}
    \|v_h\|_{H^1(\Omega;\cT_h)},
\end{equation}
where $C$ is independent of $h$.
\end{lem}


\begin{remark}
It follows from \cite{brenner2023symmetric} 
that, for every $v\in H^1_0(\Omega)$ with $\cL v\in L^2(\Omega)$,
\begin{equation}\label{eqn:remark_broken}
    |v|_{H^1(\Omega;\cT_h)}
    \leq
    C\bigl(
    |v|_{H^1(\Omega)}
    +
    h^\alpha \|\cL v\|_{L^2(\Omega)}
    \bigr),
\end{equation}
where $\alpha\in (1/2,1]$ is the elliptic regularity index \cite{dauge1988elliptic}.
\end{remark}

We also introduce the combined norm used in the interpolation and projection estimates below. It contains both the DG energy contribution and the measure-tracking contribution.

Define
\begin{equation}\label{eqn:energy_norm_h}
    \|v\|^2_h
    :=
    \beta |v|^2_{H^1(\Omega;\cT_h)}
    +
    \|v\|^2_{L^2(\Omega;\nu)}.
\end{equation}

\subsection{Operators}
\label{sec:operators}

To derive error estimates for the continuous optimal variables in $V_h$, we introduce two operators: the nodal interpolation operator $\Pi_h$ and the Ritz projection operator $\mathfrak{R}_h$ associated with $a_h^\kappa$.

\subsubsection{The interpolation operator}
\label{subsection:interpolation}


Let $\Pi_T$ denote the $\mathbb{Q}_1$ nodal interpolation operator on $T$. Then the standard interpolation estimate gives
\begin{equation}\label{standard_interp_error}
    \|v-\Pi_T v\|_{L^2(T)}
    + h_T|v-\Pi_T v|_{H^1(T)}
    + h_T^2|v-\Pi_T v|_{H^2(T)}
    \leq
    Ch_T^2\|v\|_{H^2(T)}
    \qquad
    \forall\,v\in H^2(T),
\end{equation}
see, for example, \cite{brenner2008mathematical,ciarlet2002finite,riviere2008discontinuous}.

Let $\tilde V_h$ denote the conforming $\mathbb{Q}_1$ finite element subspace associated with $\cT_h$. For sufficiently regular $v$, we define $\Pi_h v\in \tilde V_h$ elementwise by
\begin{equation}\label{eqn:global_interpolant_def}
    (\Pi_h v)|_T
    :=
    \Pi_T(v|_T)
    \qquad
    \forall T\in\cT_h .
\end{equation}
Recall that
\[
\cL v:=-\nabla\cdot(\vK\nabla v).
\]
Combining the elementwise interpolation estimate with the elliptic regularity estimates used in Section~\ref{sec:Regularities}, we use the following global interpolation bounds:
\begin{equation}\label{eqn:interp_full}
    \|v-\Pi_hv\|_{L^2(\Omega)}
    +
    h|v-\Pi_h v|_{H^1(\Omega;\cT_h)}
    \leq
    Ch^{1+\alpha}\|\cL v\|_{L^2(\Omega)}
    \qquad
    \forall v\in H^1_0(\Omega),\ \cL v\in L^2(\Omega).
\end{equation}
For the point-observation part of the tracking measure, we also use the corresponding maximum-norm estimate
\begin{equation}\label{eqn:interp_Linfty}
    \|v-\Pi_hv\|_{L^\infty(\Omega)}
    \leq
    Ch^\alpha\|\cL v\|_{L^2(\Omega)},
\end{equation} 
which is valid under the same regularity assumptions; see, for example, \cite{brenner2009multigrid,brenner2023symmetric,riviere2008discontinuous}.
Consequently,
\begin{equation}\label{est:L_interpol_energy}
    \|v-\Pi_hv\|_h
    \leq
    Ch^\alpha\|\cL v\|_{L^2(\Omega)} .
\end{equation}

\subsubsection{The Ritz projection operator}
\label{subsection:ritz_projection}


We define $\mathfrak{R}_h:H^1_0(\Omega)\to V_h$ by
\begin{equation}\label{operator_R}
    a_h^\kappa(\mathfrak{R}_hv, w_h)
    =
    a_h^\kappa(v,w_h)
    \qquad
    \forall\, w_h\in V_h .
\end{equation}
Here the right-hand side is understood by using the consistent SIPG extension of $a_h^\kappa$ to sufficiently regular continuous functions.

\begin{lem}\label{lem:ritz}
For all $v\in H^1_0(\Omega)$ with $\cL v\in L^2(\Omega)$, the following estimates hold:
\begin{align}
|v-\mathfrak{R}_hv|_{H^1(\Omega;\cT_h)}
&\leq
C h^\alpha\|\cL v\|_{L^2(\Omega)},
\label{est:operator_R1}\\
\|v-\mathfrak{R}_hv\|_{L^2(\Omega)}
&\leq
C h^{2\alpha}\|\cL v\|_{L^2(\Omega)},
\label{est:operator_R2}\\
\|v-\mathfrak{R}_hv\|_{L^2(\Omega;\nu)}
&\leq
C h^{\min\{\alpha,2\alpha-1\}}\|\cL v\|_{L^2(\Omega)}.
\label{est:operator_R3}
\end{align}
\end{lem}

\begin{proof}
We first prove the energy estimate. By coercivity, the definition of $\mathfrak{R}_h$, and continuity of $a_h^\kappa(\cdot,\cdot)$,
\begin{align*}
|\Pi_hv-\mathfrak{R}_hv|^2_{H^1(\Omega;\cT_h)}
&\leq
C a_h^\kappa(\Pi_hv-\mathfrak{R}_hv,\Pi_hv-\mathfrak{R}_hv)\\
&=
C a_h^\kappa(\Pi_hv-v,\Pi_hv-\mathfrak{R}_hv)\\
&\leq
C|\Pi_hv-v|_{H^1(\Omega;\cT_h)}
|\Pi_hv-\mathfrak{R}_hv|_{H^1(\Omega;\cT_h)} .
\end{align*}
Therefore,
\[
|\Pi_hv-\mathfrak{R}_hv|_{H^1(\Omega;\cT_h)}
\leq
C|\Pi_hv-v|_{H^1(\Omega;\cT_h)} .
\]
Using \eqref{eqn:interp_full} and the triangle inequality gives \eqref{est:operator_R1}.

Next we prove the $L^2$ estimate by a duality argument. Let $\phi\in H^1_0(\Omega)$ solve
\begin{equation}\label{eqn:dual_problem_ritz}
\cL\phi
=
v-\mathfrak{R}_hv
\quad\text{in }\Omega,
\qquad
\phi=0
\quad\text{on }\partial\Omega .
\end{equation}
Then
\[
\|\cL\phi\|_{L^2(\Omega)}
=
\|v-\mathfrak{R}_hv\|_{L^2(\Omega)} .
\]
By adjoint consistency of the SIPG method and the definition of the Ritz projection,
\begin{align*}
\|v-\mathfrak{R}_hv\|^2_{L^2(\Omega)}
&=
(\cL\phi,v-\mathfrak{R}_hv)_{L^2(\Omega)}\\
&=
a_h^\kappa(\phi,v-\mathfrak{R}_hv)\\
&=
a_h^\kappa(\phi-\Pi_h\phi,v-\mathfrak{R}_hv).
\end{align*}
Using continuity, \eqref{eqn:interp_full}, and \eqref{est:operator_R1}, we obtain
\begin{align*}
\|v-\mathfrak{R}_hv\|^2_{L^2(\Omega)}
&\leq
C|\phi-\Pi_h\phi|_{H^1(\Omega;\cT_h)}
|v-\mathfrak{R}_hv|_{H^1(\Omega;\cT_h)}\\
&\leq
C h^\alpha\|\cL\phi\|_{L^2(\Omega)}
h^\alpha\|\cL v\|_{L^2(\Omega)}\\
&=
C h^{2\alpha}
\|v-\mathfrak{R}_hv\|_{L^2(\Omega)}
\|\cL v\|_{L^2(\Omega)} .
\end{align*}
This proves \eqref{est:operator_R2}.

It remains to prove the measure-norm estimate \eqref{est:operator_R3}. Let
\[
e_h:=v-\mathfrak{R}_hv .
\]
We estimate separately the point, curve, and subdomain contributions in the definition of $\nu$.

For the point contribution, write
\[
e_h=(v-\Pi_hv)+(\Pi_hv-\mathfrak{R}_hv).
\]
To handle the second term of the right hand side, the inverse inequality and the $L^2$ estimates give
\begin{align*}
\|\Pi_hv-\mathfrak{R}_hv\|_{L^\infty(\Omega)}
&\leq
Ch^{-1}\|\Pi_hv-\mathfrak{R}_hv\|_{L^2(\Omega)}\\
&\leq
Ch^{-1}
\left(
\|\Pi_hv-v\|_{L^2(\Omega)}
+
\|v-\mathfrak{R}_hv\|_{L^2(\Omega)}
\right)\\
&\leq
C h^{-1}
\left(
h^{1+\alpha}
+
h^{2\alpha}
\right)
\|\cL v\|_{L^2(\Omega)}\\
&\leq
C h^{\min\{\alpha,2\alpha-1\}}
\|\cL v\|_{L^2(\Omega)} .
\end{align*}
Combining this with \eqref{eqn:interp_Linfty}, we obtain
\begin{equation}\label{eqn:ritz_Linfty_bound}
\|e_h\|_{L^\infty(\Omega)}
\leq
C h^{\min\{\alpha,2\alpha-1\}}
\|\cL v\|_{L^2(\Omega)} .
\end{equation}

For the subdomain contribution, \eqref{est:operator_R2} gives
\begin{equation}\label{eqn:ritz_subdomain_bound}
\|e_h\|_{L^2(\mathscr{E}_m)}
\leq
\|e_h\|_{L^2(\Omega)}
\leq
C h^{2\alpha}\|\cL v\|_{L^2(\Omega)} .
\end{equation}

For each curve contribution, we use the scaled trace inequality on the elements intersected by $\mathscr{C}_l$. Under the mesh-compatibility assumption for the observation sets,
\begin{equation}\label{eqn:curve_scaled_trace}
\|e_h\|_{L^2(\mathscr{C}_l)}
\leq
C\left(
h^{-1/2}\|e_h\|_{L^2(\Omega)}
+
h^{1/2}|e_h|_{H^1(\Omega;\cT_h)}
\right).
\end{equation}
Using \eqref{est:operator_R1} and \eqref{est:operator_R2}, we find
\begin{align}
\|e_h\|_{L^2(\mathscr{C}_l)}
&\leq
C\left(
h^{-1/2}h^{2\alpha}
+
h^{1/2}h^\alpha
\right)
\|\cL v\|_{L^2(\Omega)}
\nonumber\\
&\leq
C h^{\min\{\alpha+1/2,2\alpha-1/2\}}
\|\cL v\|_{L^2(\Omega)},
\label{eqn:ritz_curve_bound}
\end{align}
where we used $\alpha>1/2$.

Combining \eqref{eqn:ritz_Linfty_bound}, \eqref{eqn:ritz_subdomain_bound}, and \eqref{eqn:ritz_curve_bound}, together with the boundedness of the weights in \eqref{radon}, yields
\[
\|v-\mathfrak{R}_hv\|_{L^2(\Omega;\nu)}
\leq
C h^{\min\{\alpha,2\alpha-1\}}
\|\cL v\|_{L^2(\Omega)} .
\]
This proves \eqref{est:operator_R3}.
\end{proof}

\subsection{Error Estimates}
\label{sec:error_estimates}

In this subsection we establish the main a priori error estimate. We first bound the control error in $L^2(\Omega)$ by combining the continuous and discrete variational inequalities with the projection properties of the admissible control set. We then recover the state error in $L^2(\Omega)$ from the control error. The reduced regularity of the adjoint state enters through the measure-tracking estimates and through the adjoint-coupling estimate stated below.

We use the notation $S_h\xi\in V_h$ for the discrete state associated with any source term $\xi\in L^2(\Omega)$, namely
\begin{equation}\label{eqn:extended_discrete_state_map}
a_h^\kappa(S_h\xi,v_h)
=
\iO (f+\xi)v_h\,d\vx
\qquad
\forall v_h\in V_h .
\end{equation}

\begin{lem}\label{lem:state_from_control}
Let $\bgm,\bgm_h$ be the continuous and discrete optimal controls, with states $\bp=S\bgm$ and $\bp_h=S_h\bgm_h$. Then
\begin{equation}\label{eqn:state_from_control}
    \|\bp-\bp_h\|_{L^2(\Omega)}
    \leq
    \|\bp-S_h\bgm\|_{L^2(\Omega)}
    +
    C\,\|\bgm-\bgm_h\|_{L^2(\Omega)}.
\end{equation}
Moreover,
\begin{equation}\label{eqn:fixed_control_state_error}
    \|\bp-S_h\bgm\|_{L^2(\Omega)}
    \leq
    Ch^{2\alpha}\|\cL\bp\|_{L^2(\Omega)} .
\end{equation}
\end{lem}

\begin{proof}
By linearity of the discrete state equation,
\[
S_h\bgm-\bp_h=S_h\bgm-S_h\bgm_h
\]
satisfies
\begin{equation*}
a_h^\kappa(S_h\bgm-\bp_h,v_h)
=
\iO(\bgm-\bgm_h)v_h\,d\vx
\qquad
\forall v_h\in V_h .
\end{equation*}
Therefore Lemma~\ref{lem:discrete_bound} gives
\[
\|S_h\bgm-\bp_h\|_{L^2(\Omega)}
\leq
C\|\bgm-\bgm_h\|_{L^2(\Omega)}.
\]
The triangle inequality gives \eqref{eqn:state_from_control}.

To prove~\eqref{eqn:fixed_control_state_error}, note that $\bar{p} \in H^2(\Omega) \cap H_0^1(\Omega)$ is continuous 
across element interfaces, so $\llbracket \bar{p} \rrbracket = 0$ on all edges. 
The penalty and symmetry terms in $a_h^\kappa$ therefore vanish when $\bar{p}$ is 
inserted, and elementwise integration by parts gives
\begin{equation*}
    a_h^\kappa(\bar{p},\, v_h) 
    = \int_\Omega (f + \bar{\gamma})\, v_h \,d\mathbf{x} 
    \qquad \forall\, v_h \in V_h.
\end{equation*}
By \eqref{operator_R} and \eqref{eqn:extended_discrete_state_map},
we have $\mathfrak{R}_h \bar{p} = S_h\bar{\gamma}$, and hence 
estimate~\eqref{eqn:fixed_control_state_error} then follows from~\eqref{est:operator_R2}.
\end{proof}
\begin{lem}\label{lem:adjoint_coupling}
Let
\[
\widetilde{\gamma}_h:=Q_h\bgm,
\qquad
\widetilde p_h:=S_h\widetilde{\gamma}_h .
\]
Then, for every $\epsilon>0$,
\begin{equation}\label{eqn:adjoint_coupling_estimate}
(\bz_h-\bz,\widetilde{\gamma}_h-\bgm_h)_{L^2(\Omega)}
\leq
C h^{\min\{\alpha,2\alpha-1,\,1-\epsilon\}}
\|\mathcal{L}_h(\widetilde p_h-\bp_h)\|_{L^2(\Omega)} .
\end{equation}
\end{lem}
\begin{proof}
Set $\varphi := \tilde{p}_h - \bar{p}_h \in V_h$.
Since $a_h^\kappa(\tilde{p}_h - \bar{p}_h, v_h)
= \int_\Omega (\tilde{\gamma}_h - \bar{\gamma}_h)\,v_h\,d\mathbf{x}$
for all $v_h \in V_h$, the definition \eqref{discrete_L} of $\mathcal{L}_h$ gives
\begin{equation*}
    \mathcal{L}_h \varphi = \tilde{\gamma}_h - \bar{\gamma}_h,
\end{equation*}
so the claimed estimate is equivalent to
\begin{equation*}
    (\bar{z}_h - \bar{z}, \mathcal{L}_h\varphi)_{L^2(\Omega)}
    \leq C h^{\min\{\alpha,2\alpha-1,\,1-\epsilon\}}
    \|\mathcal{L}_h\varphi\|_{L^2(\Omega)}.
\end{equation*}

Define the auxiliary discrete function $\hat{z}_h \in V_h$ by
\begin{equation}\label{aux_zhat}
    a_h^\kappa(v_h, \hat{z}_h)
    = (v_h, \bar{p} - p_d)_{L^2(\Omega;\nu)}
    \quad \forall\, v_h \in V_h,
\end{equation}
and split
\begin{equation*}
    (\bar{z}_h - \bar{z},\, \mathcal{L}_h\varphi)_{L^2(\Omega)}
    = \underbrace{(\bar{z}_h - \hat{z}_h,\, \mathcal{L}_h\varphi)_{L^2(\Omega)}}_{\mathrm{(I)}}
    + \underbrace{(\hat{z}_h - \bar{z},\, \mathcal{L}_h\varphi)_{L^2(\Omega)}}_{\mathrm{(II)}}.
\end{equation*}

\noindent\textbf{Estimate of (I).}
Subtracting the equation for $\hat{z}_h$ from the discrete adjoint
equation~\eqref{discrete_optimality_condition2} gives
\begin{equation*}
    a_h^\kappa(v_h,\, \bar{z}_h - \hat{z}_h)
    = (v_h,\, \bar{p}_h - \bar{p})_{L^2(\Omega;\nu)}
    \quad \forall\, v_h \in V_h.
\end{equation*}
By the definition of $\mathcal{L}_h$ \eqref{discrete_L}, the symmetry of $a_h^\kappa$,
and taking $v_h = \varphi$,
\begin{equation*}
    (\bar{z}_h - \hat{z}_h,\, \mathcal{L}_h\varphi)_{L^2(\Omega)}
    = a_h^\kappa (\varphi,\, \bar{z}_h - \hat{z}_h)
    = (\bar{p}_h - \bar{p},\, \varphi)_{L^2(\Omega;\nu)}.
\end{equation*}
Writing $\bar{p}_h - \bar{p} = (\bar{p}_h - \tilde{p}_h) + (\tilde{p}_h - \bar{p})
= -\varphi + (\tilde{p}_h - \bar{p})$ yields
\begin{equation*}
    (\bar{p}_h - \bar{p},\, \varphi)_{L^2(\Omega;\nu)}
    = -\|\varphi\|_{L^2(\Omega;\nu)}^2
      + (\tilde{p}_h - \bar{p},\, \varphi)_{L^2(\Omega;\nu)}.
\end{equation*}
Since the first term is non-positive, we obtain
\begin{equation}\label{estI:aux1}
    \mathrm{(I)} \leq (\tilde{p}_h - \bar{p},\, \varphi)_{L^2(\Omega;\nu)}.
\end{equation}
To bound the right-hand side, let $w := G(\mathcal{L}_h\varphi) \in H^2(\Omega)\cap H_0^1(\Omega)$
denote the solution of
\begin{equation}\label{aux_darcy_sol}
    -\nabla\cdot(K\nabla w) = \mathcal{L}_h\varphi.
\end{equation}
By the elliptic regularity,
\begin{equation}\label{aux_elliptic}
    \|w\|_{H^2(\Omega)} \leq C\|\mathcal{L}_h\varphi\|_{L^2(\Omega)}.
\end{equation}
Also, by \eqref{operator_R}, we get
\begin{equation*}
    a_h^\kappa(\mathfrak{R}_h w, v_h)=a_h^\kappa(w, v_h)=(\mathcal{L}_h\varphi, v_h)_{L^2(\Omega)}=a_h^\kappa(\varphi, v_h).
\end{equation*}
Thus
$\varphi = \mathfrak{R}_h w = w + \mathfrak{R}_h w-w$, and hence
by the triangle inequality, continuous embedding $H^2(\Omega) \hookrightarrow L^\infty(\Omega)$, \eqref{est:operator_R3}, \eqref{aux_darcy_sol}, and \eqref{aux_elliptic}, 
\begin{equation}\label{estI:aux2}
    \|\varphi\|_{L^2(\Omega;\nu)}
    \leq \|w\|_{L^2(\Omega;\nu)} + \|w - \mathfrak{R}_h w\|_{L^2(\Omega;\nu)}
    \leq C\|w\|_{H^2(\Omega)} + Ch^{\min\{\alpha,2\alpha-1\}}
    \|\mathcal{L}_h\varphi\|_{L^2(\Omega)}
    \leq C\|\mathcal{L}_h\varphi\|_{L^2(\Omega)}.
\end{equation}
Now, 
\begin{equation}\label{estI_2nd}
    \|\tilde{p}_h - \bar{p}\|_{L^2(\Omega;\nu)}
    \leq \|\tilde{p}_h - \mathfrak{R}_h\bar{p}\|_{L^2(\Omega;\nu)} + \|\mathfrak{R}_h\bar{p} - \bar{p}\|_{L^2(\Omega;\nu)}.
\end{equation}
The second term of the right hand side of \eqref{estI_2nd} is estimated by
\eqref{est:operator_R3} as
\begin{equation}\label{estI:aux3}
    \|\mathfrak{R}_h\bar{p} - \bar{p}\|_{L^2(\Omega;\nu)}\leq Ch^{\min\{\alpha,2\alpha-1\}}.
\end{equation}
For the first term of \eqref{estI_2nd}, 
note that $\tilde{p}_h - \mathfrak{R}_h\bar{p}$ is thus the \emph{discrete} state driven by the source $Q_h\bar{\gamma}-\bar{\gamma}$ and is, in general, \emph{not} equal to the continuous solve $G(Q_h\bar{\gamma}-\bar{\gamma})$. To bound it, let
\begin{equation*}
    \zeta := G(Q_h\bar{\gamma}-\bar{\gamma}) \in H^2(\Omega)\cap H^1_0(\Omega)
\end{equation*}
denote the continuous solution of $-\nabla\cdot(K\nabla\zeta)=Q_h\bar{\gamma}-\bar{\gamma}$. Since $\zeta\in H^2(\Omega)\cap H^1_0(\Omega)$, SIPG consistency together with \eqref{operator_R} gives
\begin{equation}\label{eqn:estI_aux_Ritz1}
    a_h^\kappa(\mathfrak{R}_h\zeta,v_h)=a_h^\kappa(\zeta,v_h)=(Q_h\bar{\gamma}-\bar{\gamma},v_h)_{L^2(\Omega)}
\end{equation}
for all $v_h\in V_h$. 
On the other hand, 
we can rewrite
\begin{equation*}
    \tilde{p}_h - \mathfrak{R}_h\bar{p}=S_h(Q_h\bar{\gamma}) - S_h\bar{\gamma}
    =S_h(Q_h\bar{\gamma} - \bar{\gamma})
\end{equation*}
because $a_h^\kappa(\mathfrak{R}_h\bar{p},v_h)=(f+\bar{\gamma},v_h)_{L^2(\Omega)}=a_h^\kappa(S_h\bar{\gamma},v_h)$ for any $v_h\in V_h$.
Then 
\begin{equation}\label{eqn:estI_aux_Ritz2}
    a_h^\kappa(\tilde{p}_h - \mathfrak{R}_h\bar{p},v_h)=(Q_h\bar{\gamma}-\bar{\gamma},v_h)_{L^2(\Omega)}.
\end{equation}
Using coercivity of $a_h^\kappa$, \eqref{eqn:estI_aux_Ritz1}, and \eqref{eqn:estI_aux_Ritz2}, it follows
$\tilde{p}_h - \mathfrak{R}_h\bar{p} = \mathfrak{R}_h\zeta$.
Hence, by the triangle inequality, the embedding $H^2(\Omega)\hookrightarrow L^\infty(\Omega)$ with the finiteness of $\nu$, the elliptic regularity $\|\zeta\|_{H^2(\Omega)}\le C\|Q_h\bar{\gamma}-\bar{\gamma}\|_{L^2(\Omega)}$, \eqref{est:operator_R3}, and the control approximation estimate \eqref{eqn:control_projection_estimate},
\begin{equation}\label{estI:aux4}
\begin{aligned}
    \|\tilde{p}_h - \mathfrak{R}_h\bar{p}\|_{L^2(\Omega;\nu)}
    &= \|\mathfrak{R}_h\zeta\|_{L^2(\Omega;\nu)}
    \leq \|\zeta\|_{L^2(\Omega;\nu)} + \|\zeta - \mathfrak{R}_h\zeta\|_{L^2(\Omega;\nu)} \\
    &\leq C\|\zeta\|_{H^2(\Omega)}
        + Ch^{\min\{\alpha,2\alpha-1\}}\|Q_h\bar{\gamma}-\bar{\gamma}\|_{L^2(\Omega)} \\
    &\leq C\|Q_h\bar{\gamma}-\bar{\gamma}\|_{L^2(\Omega)}
    \leq Ch^{1-\epsilon}.
\end{aligned}
\end{equation}
Combining \eqref{estI:aux1}, \eqref{estI:aux2}, \eqref{estI_2nd}, \eqref{estI:aux3}, and \eqref{estI:aux4}, we have
\begin{equation}\label{estI}
    \mathrm{(I)}\leq Ch^{\min\{\alpha,2\alpha-1,\,1-\epsilon\}}
\|\mathcal{L}_h\varphi\|_{L^2(\Omega)} .
\end{equation}

\noindent\textbf{Estimate of (II).}
By the definition of $\mathcal{L}_h$ \eqref{discrete_L}, the symmetry of $a_h^\kappa$,
and the equation \eqref{aux_zhat} with $v_h = \varphi$,
\begin{equation*}
    (\hat{z}_h,\, \mathcal{L}_h\varphi)_{L^2(\Omega)}
    = a_h^\kappa (\varphi,\, \hat{z}_h)
    = (\bar{p} - p_d,\, \varphi)_{L^2(\Omega;\nu)}.
\end{equation*}
Since $\mathcal{L}_h\varphi \in L^2(\Omega)$, the transposition identity~\eqref{eqn:transposition_identity} with $w$
as above gives
\begin{equation*}
    (\bar{z},\, \mathcal{L}_h\varphi)_{L^2(\Omega)}
    = (\bar{p} - p_d,\, w)_{L^2(\Omega;\nu)}.
\end{equation*}
Subtracting,
\begin{equation*}
    \mathrm{(II)}
    = (\bar{p} - p_d,\, \varphi - w)_{L^2(\Omega;\nu)}.
\end{equation*}
Recall that $\varphi -w= \mathfrak{R}_h w - w$.
Applying the Cauchy--Schwarz inequality, \eqref{est:operator_R3}, and \eqref{aux_darcy_sol},
\begin{equation}\label{estII}
    \mathrm{(II)}
    \leq \|\bar{p} - p_d\|_{L^2(\Omega;\nu)}\,\|w - \mathfrak{R}_h w\|_{L^2(\Omega;\nu)}
    \leq Ch^{\min\{\alpha,2\alpha-1\}}
    \|\mathcal{L}w\|_{L^2(\Omega)}
    = Ch^{\min\{\alpha,2\alpha-1\}}
    \|\mathcal{L}_h\varphi\|_{L^2(\Omega)}.
\end{equation}

Combining \eqref{estI} and \eqref{estII},
\begin{equation*}
    (\bar{z}_h - \bar{z},\; \tilde{\gamma}_h - \bar{\gamma}_h)_{L^2(\Omega)}
    \leq Ch^{\min\{\alpha,2\alpha-1,\,1-\epsilon\}}
    \|\mathcal{L}_h(\tilde{p}_h - \bar{p}_h)\|_{L^2(\Omega)},
\end{equation*}
which completes the proof.
\end{proof}

\begin{thm}\label{thm:error}
For every $\epsilon > 0$, there exists a constant $C_\beta > 0$, 
independent of $h$, such that
\begin{equation}
    \|\bar{\gamma} - \bar{\gamma}_h\|_{L^2(\Omega)} 
    \leq C_\beta\, h^{\min\{\theta,\, 1-\epsilon\}},
\end{equation}
where $\theta = \min\{\alpha, 2\alpha - 1\}$ with $\alpha \in (1/2, 1]$.
Moreover, there exists $C'_\beta > 0$, independent of $h$, such that
\begin{equation}
    \|\bar{p} - \bar{p}_h\|_{L^2(\Omega)} 
    \leq C'_\beta\, h^{\min\{\theta,\, 1-\epsilon\}}.
\end{equation}
\end{thm}

\begin{proof}
We first prove the control estimate and then recover the state estimate from Lemma~\ref{lem:state_from_control}.

Recall
$\widetilde{\gamma}_h:=Q_h\bgm$ and $\widetilde p_h:=S_h\widetilde{\gamma}_h$.
Then
\begin{equation}\label{pf:thm:00}
    a_h^\kappa(\widetilde p_h,v_h)
    =
    \iO (f+\widetilde{\gamma}_h)v_h\,d\vx
    \qquad
    \forall\,v_h\in V_h .
\end{equation}

We start with 
\begin{equation}\label{pf:thm:01}
\beta\|\bgm-\bgm_h\|_{L^2(\Omega)}^2
=
\beta(\bgm-\bgm_h,\bgm-\widetilde{\gamma}_h)_{L^2(\Omega)}
+
\beta(\bgm-\bgm_h,\widetilde{\gamma}_h-\bgm_h)_{L^2(\Omega)} .
\end{equation}
Note that by \eqref{eqn:control_constraints} and the positivity of $Q_h$ in \eqref{Q_h_positive}, 
$Q_h\gamma_-|_T
\leq
\widetilde{\gamma}_h|_T
\leq
Q_h\gamma_+|_T$ for all $T\in\cT_h$,
and hence $\widetilde{\gamma}_h\in U_{\gamma,h}$.

For the second term in \eqref{pf:thm:01}, 
by 
\eqref{discrete_optimality_condition3} with $\eta_h=\widetilde{\gamma}_h$, we have
\[
(\beta\bgm_h+\bz_h,\widetilde{\gamma}_h-\bgm_h)_{L^2(\Omega)}
\geq 0.
\]
Therefore,
\begin{equation}\label{pf:thm:01b}
\begin{aligned}
\beta(\bgm-\bgm_h,\widetilde{\gamma}_h-\bgm_h)_{L^2(\Omega)}
&\leq
(\beta\bgm+\bz_h,\widetilde{\gamma}_h-\bgm_h)_{L^2(\Omega)} .
\end{aligned}
\end{equation}
Writing $\lambda=\beta\bgm+\bz$, we obtain
\begin{equation}\label{pf:thm:01c}
(\beta\bgm+\bz_h,\widetilde{\gamma}_h-\bgm_h)_{L^2(\Omega)}
=
(\lambda,\widetilde{\gamma}_h-\bgm_h)_{L^2(\Omega)}
+
(\bz_h-\bz,\widetilde{\gamma}_h-\bgm_h)_{L^2(\Omega)} .
\end{equation}

We first estimate the first term in \eqref{pf:thm:01c}. Using \eqref{L_decomposition}, \eqref{L_minmax}
\eqref{eqn:pointwise_complementarity}, and the discrete constraints
\[Q_h\gamma_-
\leq
\bgm_h
\leq
Q_h\gamma_+,\]
we obtain
\begin{align}\label{pf:thm:multiplier_decomp}
(\lambda,\widetilde{\gamma}_h-\bgm_h)_{L^2(\Omega)}
&\leq
(\lambda^+,Q_h\bgm-\bgm)_{L^2(\Omega)}
+
(\lambda^-,Q_h\bgm-\bgm)_{L^2(\Omega)}
\nonumber\\
&\quad
+
(\lambda^+,\gamma_- - Q_h\gamma_-)_{L^2(\Omega)}
+
(\lambda^-,\gamma_+ - Q_h\gamma_+)_{L^2(\Omega)} .
\end{align}
Since $Q_h$ is the elementwise $L^2$ projection and $Q_h\lambda^\pm\in W_h$, the orthogonality of $Q_h$ gives
\begin{align}\label{pf:thm:multiplier_orthogonality}
(\lambda,\widetilde{\gamma}_h-\bgm_h)_{L^2(\Omega)}
&\leq
(\lambda^+-Q_h\lambda^+,Q_h\bgm-\bgm)_{L^2(\Omega)}
+
(\lambda^--Q_h\lambda^-,Q_h\bgm-\bgm)_{L^2(\Omega)}
\nonumber\\
&\quad
+
(\lambda^+-Q_h\lambda^+,\gamma_- - Q_h\gamma_-)_{L^2(\Omega)}
+
(\lambda^--Q_h\lambda^-,\gamma_+ - Q_h\gamma_+)_{L^2(\Omega)} .
\end{align}
By \eqref{eqn:lambda_pm_regular}, \eqref{control_regul}, the regularity of $\gamma_\pm$, and the approximation property of the elementwise $L^2$ projection, for every $\epsilon>0$,
\begin{equation}\label{eqn:Qh_projection_near_first_order}
\|\xi-Q_h\xi\|_{L^2(\Omega)}
\leq
C h^{1-\epsilon}
\qquad
\text{for } \xi\in\{\bgm,\gamma_-,\gamma_+,\lambda^+,\lambda^-\}.
\end{equation}
Consequently, using Cauchy-Schwartz inequality,
\begin{equation}\label{pf:thm:03}
(\lambda,\widetilde{\gamma}_h-\bgm_h)_{L^2(\Omega)}
\leq
C h^{2-2\epsilon}.
\end{equation}

Next, by Lemma~\ref{lem:adjoint_coupling},
\begin{equation}\label{pf:thm:04}
\left|
(\bz_h-\bz,\widetilde{\gamma}_h-\bgm_h)_{L^2(\Omega)}
\right|
\leq
C h^{\min\{\theta,1-\epsilon\}}
\|\mathcal{L}_h(\widetilde p_h-\bp_h)\|_{L^2(\Omega)} .
\end{equation}
By the definition of $\mathcal{L}_h$ \eqref{discrete_L},  \eqref{discrete_optimality_condition1}, and \eqref{pf:thm:00},
\[
\mathcal{L}_h(\widetilde p_h-\bp_h)
=
\widetilde{\gamma}_h-\bgm_h
=
Q_h\bgm-\bgm_h .
\]
Thus, by \eqref{eqn:Qh_projection_near_first_order},
\begin{equation}\label{pf:thm:07}
\|\mathcal{L}_h(\widetilde p_h-\bp_h)\|_{L^2(\Omega)}
\leq
\|Q_h\bgm-\bgm\|_{L^2(\Omega)}
+
\|\bgm-\bgm_h\|_{L^2(\Omega)}
\leq
C h^{1-\epsilon}
+
\|\bgm-\bgm_h\|_{L^2(\Omega)} .
\end{equation}

The first term on the right-hand side of \eqref{pf:thm:01} is estimated by Cauchy--Schwarz and \eqref{eqn:Qh_projection_near_first_order}:
\begin{equation}\label{pf:thm:08}
\beta
(\bgm-\bgm_h,\bgm-\widetilde{\gamma}_h)_{L^2(\Omega)}
\leq
C_\beta h^{1-\epsilon}
\|\bgm-\bgm_h\|_{L^2(\Omega)} .
\end{equation}

Combining \eqref{pf:thm:01}--\eqref{pf:thm:08}, we obtain
\begin{equation}\label{pf:thm:09}
\beta\|\bgm-\bgm_h\|_{L^2(\Omega)}^2
\leq
C_\beta h^{1-\epsilon}
\|\bgm-\bgm_h\|_{L^2(\Omega)}
+
C h^{2-2\epsilon}
+
C h^{\min\{\theta,1-\epsilon\}}
\left(
h^{1-\epsilon}
+
\|\bgm-\bgm_h\|_{L^2(\Omega)}
\right).
\end{equation}
Using the arithmetic--geometric mean inequality, the terms involving
$\|\bgm-\bgm_h\|_{L^2(\Omega)}$ can be absorbed into the left-hand side. Hence
\begin{equation}\label{pf:thm:control_final}
\|\bgm-\bgm_h\|_{L^2(\Omega)}
\leq
C_\beta h^{\min\{\theta,1-\epsilon\}} .
\end{equation}

Finally, Lemma~\ref{lem:state_from_control}, \eqref{eqn:fixed_control_state_error}, and \eqref{pf:thm:control_final} give
\begin{equation*}
\|\bp-\bp_h\|_{L^2(\Omega)}
\leq
Ch^{2\alpha}\|\cL\bp\|_{L^2(\Omega)}
+
C\|\bgm-\bgm_h\|_{L^2(\Omega)}\leq
C'_\beta h^{\min\{\theta,1-\epsilon\}}. \qedhere
\end{equation*}
\end{proof}

\section{Numerical Experiments}
\label{sec:numerics}
We now present numerical experiments confirming the convergence rates of Theorem~\ref{thm:error},
demonstrating the local mass conservation established in Section~\ref{sec:local_mass_conservation}, and
illustrating the robustness of the coefficient-weighted SIPG discretization for heterogeneous
permeability fields. All computations were performed using the \texttt{deal.II} finite element
library~\cite{dealii}.

\subsection{Numerical Implementation}
Let $\{\phi^p_1,\dots,\phi^p_{n}\}$ be a finite set of piecewise linear basis functions for the state $p$,
$\{\phi^z_1,\dots,\phi^z_{n}\}$ for the adjoint state $z$,
and
$\{\psi_1,\dots,\psi_{m}\}$ be a finite set of piecewise constant basis functions for the control $\gamma$.
Then
\[ p(\vx) = \sum^{n}_i p_i\phi^p_i,\quad
z(\vx) = \sum^{n}_i z_i\phi^z_i,\quad\text{and}\quad
\gamma(\vx) = \sum^{m}_i \gamma_i\psi_i.\]
In terms of the unknown vectors 
$\overrightarrow{p} = (p_1,\dots,p_{n})^\top,\;
\overrightarrow{z} = (z_1,\dots,z_{n})^\top,\;
\overrightarrow{\gamma} = (\gamma_1,\dots,\gamma_{m})^\top$,
the associated optimality system reads
\begin{subequations}
\begin{alignat}{1}
&A_\vK\overrightarrow{p} - M_{p,\gamma}\overrightarrow{\gamma} = \overrightarrow{f}\\
&A_\vK\overrightarrow{z} - R\overrightarrow{p} = - \overrightarrow{p_d},\\
&(\beta M_\gamma\overrightarrow{\gamma} + M_{p,\gamma}^\top\overrightarrow{z})^\top(\overrightarrow{\eta} - \overrightarrow{\gamma})\geq 0,
\end{alignat}
\end{subequations}
for any $\overrightarrow{\gamma}_-\leq\overrightarrow{\eta}\leq\overrightarrow{\gamma}_+$.
Here
\[
A_\vK(i,j) = a_h^\kappa(\phi_j,\phi_i),\quad
M(i,j) = \iO \phi_j\phi_i\,d\vx,\quad
R(i,j) = \iO \phi_j\phi_i\,d\nu,
\]
and
\[
\overrightarrow{f}(i) = \iO f\phi_i\,d\vx,\quad
\overrightarrow{p_d}(i) = \iO p_d\phi_i\,d\nu.
\]

Since $\psi_i$ are piecewise constant, $M_\gamma$ is a diagonal matrix and so
\[
\overrightarrow{\mu}:=-\left( \overrightarrow{\gamma} + (\beta M_\gamma)^{-1}M_{p,\gamma}^\top\overrightarrow{z}\right).
\]
For the components of the optimal vector $\overrightarrow{\gamma}$, we obtain
\[\gamma_i=\left\{
\begin{array}{ccl}
     \gamma_-&\text{if}&\gamma_i+\mu_i<\gamma_-,  \\
     \mu_i&\text{if}&\gamma_i+\mu_i\in[\gamma_-,\gamma_+],  \\
     \gamma_+&\text{if}&\gamma_i+\mu_i>\gamma_+ .
\end{array}\right.
\]
Finally, this leads us to use the primal-dual active set algorithm \cite{hintermuller2002primal, troltzsch2010optimal}:

\begin{itemize}
    \item[\it{1.}] Choose initial guesses $\overrightarrow{\gamma_0}$ and $\overrightarrow{\mu_0}$.
    In the $k$th step, define the active sets and inactive set as 
    \begin{align*}
    A_k^+ &= \big\{i\in\{1,\dots,n_\gamma\}:\gamma_{k-1,i} + \mu_{k-1,i}>\gamma_+\big\},\\
    A_k^- &= \big\{i\in\{1,\dots,n_\gamma\}:\gamma_{k-1,i} + \mu_{k-1,i}<\gamma_-\big\},\\
    I_k &= \{1,\dots,n_\gamma\}\setminus (A_k^+ \cup A_k^-).
    \end{align*}
    Let $\chi^+_k$ and $\chi^-_k$ denote the characteristic matrices of $A^+_k$ and $A^-_k$, respectively, with the diagonal elements
    \[
    \chi_{k}^+(i,i)=\left\{\begin{array}{cl}
        1 & \text{if}\; i\in A^+_k \\
        0 & \text{otherwise},
    \end{array}\right.
    \qquad
    \chi_{k}^-(i,i)=\left\{\begin{array}{cl}
        1 & \text{if}\; i\in A^-_k \\
        0 & \text{otherwise}.
    \end{array}\right.
    \]

    \item[2.]  Set $E = (\beta M_\gamma)^{-1}(I-\chi^-_k-\chi^+_k)$. The diagonal elements $E(i,i)$ vanish if and only if $i\in A^-_k\cup A^+_k$.
    Then we solve the following system of linear equations for $\overrightarrow{p},\;\overrightarrow{z},\; \overrightarrow{\gamma}$:
    \[
    \left[\begin{array}{ccc}
       \vspace{5pt} A_\vK & 0  & -M_{p,\gamma}\\
       \vspace{5pt} -R & A_\vK & 0 \\
       \vspace{5pt} 0  & EM_{p,\gamma}^\top  & I
    \end{array}\right]
    \left[\begin{array}{c}
       \vspace{5pt} \overrightarrow{p}\\ 
       \vspace{5pt} \overrightarrow{z}\\ 
       \overrightarrow{\gamma}
    \end{array}\right]
    = \left[\begin{array}{c}
       \vspace{5pt} \overrightarrow{f}\\ \vspace{5pt} -\overrightarrow{p_d}\\ 
       \chi_k^-\overrightarrow{\gamma_-} + \chi_k^+\overrightarrow{\gamma_+}
    \end{array}\right].
    \]

    \item[\it{3.}] Put $\overrightarrow{\gamma_k} := \overrightarrow{\gamma}$ and $\overrightarrow{\mu_k} := -\left(\overrightarrow{\gamma_k} + (\beta M_\gamma)^{-1}(M_{p,\gamma}^\top\overrightarrow{z}_k)\right)$.
    Update the active sets and inactive set.
    Terminate the algorithm when $A^\pm_{k+1} = A^\pm_{k}$.
\end{itemize}

The error is computed by
\begin{equation}
    \text{error} = \|\bp_6 - \bp_k\|_{L_2(\Omega)} + \|\bgm_6 - \bgm_k\|_{L_2(\Omega)},
\end{equation}
where $\bp_6$ and $\bgm_6$ are the discrete solutions when $h=2^{-6}$ and $\bp_k$ and $\bgm_k$ are the discrete solutions when $h=2^{-k}$.

\subsection{Example 1. Multi-feature geometric tracking } 
To evaluate the performance and geometric tracking capability of the proposed method, we consider three distinct benchmark configurations over the computational domain $\Omega = (-1, 1)^2$, as illustrated in Figure~\ref{fig:ex1_setsup}.
In each case, specific target features, including internal subdomains, linear interfaces, and discrete point locations, are tracked within the global domain. 
The control regularization parameter $\beta = 1$, and the bounds for the control constraints $\gamma_+ = 30$, and $\gamma_-=-\infty$.

\begin{itemize}
    \item[(i)] {Isolated subdomain with discrete points:} 
    The subdomain $[-0.5, 0.5] \times [-0.75, -0.25]$ is situated in the lower half of the region. Inside this subdomain, two discrete points are tracked symmetrically at $(-0.25, -0.5)$ and $(0.25, -0.5)$ (Figure~\ref{fig:ex1_case i)}). The weights for tracking are $w_{\mathscr{E}}=10$, and $w_{\mathscr{P}_1}=w_{\mathscr{P}_2}=10^4$. Also, the desired state for the subdomain is $p_\mathscr{E} =6 $  and for points $p_\mathscr{P} = 4$.

    \item[(ii)] {Vertical line interface with offset points:} 
    A 1D vertical interface defined along $x = 0$ for $y \in [-0.5, 0.5]$ divides the central area. In addition, two discrete points are placed diagonally across the line at $(-0.5, 0.25)$ and $(0.5, -0.25)$ (Figure~\ref{fig:ex1_case ii)}). The weights for tracking are $w_{\mathscr{C}}=10^2$, and $w_{\mathscr{P}_1}=w_{\mathscr{P}_2}=10^3$. Also, the desired state for the line segment is $p_\mathscr{C} =6 $  and for points $p_\mathscr{P} = 4$.

    \item[(iii)] {Hybrid feature interaction:} 
    This setup combines all three geometric entities: a rectangular subdomain $[-0.5, 0.5] \times [-0.75, -0.125]$, a vertical segment along $x = 0$ for $y \in [-0.375, 0.25]$, and a single discrete point located at $(0.0, -0.5)$ (Figure~\ref{fig:ex1_case iii)}). The weights for tracking are $w_{\mathscr{E}}=1$, $w_{\mathscr{C}}=10^3$ and $w_{\mathscr{P}}=10^2$. Also, the desired state for the subdomain is $p_\mathscr{E} = 4$, for line segment $p_\mathscr{C} = 4$  and for point $p_\mathscr{P} = 4.5$.
\end{itemize}

\begin{figure}[h!]
     \centering
     \begin{subfigure}{0.2\textwidth}
         \centering
\begin{tikzpicture}[scale=2]
    \draw  (-1, -1) rectangle (1, 1);
    \draw[dashed, blue!70!black, fill=blue!20, fill opacity=0.7] (-0.5, -0.75) rectangle (0.5, -0.25);

    \filldraw[red!80!black] (-0.25, -0.5) circle (0.03) ;
    \filldraw[red!80!black] (0.25, -0.5) circle (0.03);
    
\end{tikzpicture}
\caption{Case i)}
         \label{fig:ex1_case i)}
     \end{subfigure} \hspace{50pt}
\begin{subfigure}{0.2\textwidth}
         \centering
\begin{tikzpicture}[scale=2]
    \draw (-1, -1) rectangle (1, 1);

    \draw[line width=1.pt, green!70!black] (0, -0.5) -- (0, 0.5);

    \filldraw[red!80!black] (-0.5, 0.25) circle (0.03);
    \filldraw[red!80!black] (0.5, -0.25) circle (0.03);
\end{tikzpicture}
\caption{Case ii)}
         \label{fig:ex1_case ii)}
     \end{subfigure} \hspace{50pt}
\begin{subfigure}{0.2\textwidth}
         \centering
\begin{tikzpicture}[scale=2]
    \draw (-1, -1) rectangle (1, 1);

    \draw[dashed, blue!70!black, fill=blue!15] (-0.5, -0.75) rectangle (0.5, -0.125);

    \draw[line width=1.pt, green!80!black] (0., -0.375) -- (0., 0.25);

    \filldraw[red!80!black] (-0., -0.5) circle (0.03);
    
\end{tikzpicture}
\caption{Case iii)}
         \label{fig:ex1_case iii)}
     \end{subfigure}
     \caption{Example 1. Three different cases.}
\label{fig:ex1_setsup}
\end{figure}
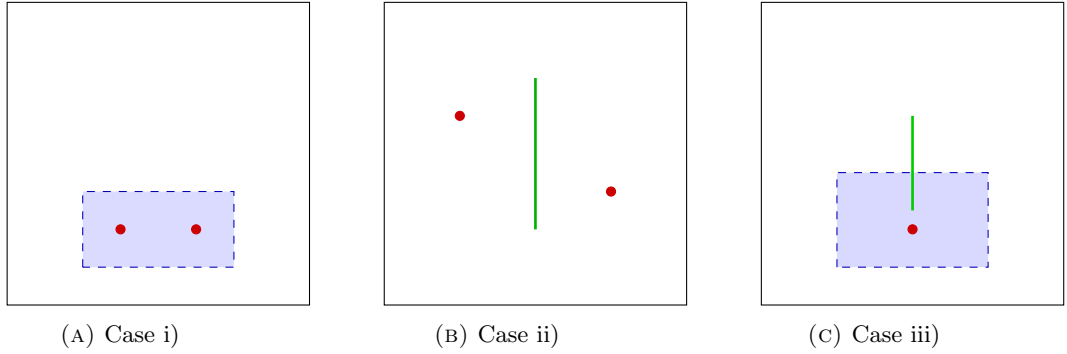

\begin{table}[!h]
\centering
\setlength{\tabcolsep}{3pt}
\begin{tabular}{c||c|c|c||c|c|c||c|c|c}
& \multicolumn{3}{c||}{Case i)}  
& \multicolumn{3}{c||}{Case ii)}  
& \multicolumn{3}{c}{Case iii)}   \\
$h$ & error & order & PDAS It 
& error & order & PDAS It 
& error & order& PDAS It \\
\hline
$2^{-2}$&1.18e+01& -   & 1 &1.17e+01& -  & 1 &9.95e+00& -  &1\\
$2^{-3}$&5.99e+00&0.97 & 3 &6.26e+00&0.90& 1 &5.31e+00&0.73&1\\
$2^{-4}$&2.67e+00&1.14 & 3 &2.77e+00&1.18& 3 &2.38e+00&1.19&2\\
$2^{-5}$&1.15e+00&1.22 & 3 &1.20e+00&1.20& 3 &1.04e+00&1.21&2
\end{tabular}
\caption{Example 1. Convergence results for each different cases. }
\label{tab:example1}
\end{table}

Table~\ref{tab:example1} reports the convergence history for the three cases.
Across all configurations, the observed convergence orders stabilize between $0.97$ and $1.22$ as the mesh is refined from $h = 2^{-2}$ to $h = 2^{-5}$, in agreement with the theoretical rate $h^{\min\{\theta,\,1-\epsilon\}}$ predicted by Theorem~\ref{thm:error}. 
The primal-dual active set algorithm converges in at most three iterations across all mesh levels and all three cases, indicating that the number of PDAS iterations does not grow with mesh refinement.

Figure~\ref{fig:ex1} illustrates the computed state and control for each case.
In Case~(i), since the point weights exceed the subdomain weight, the point objectives dominate, and the optimal pressure is pulled toward $4$ near tracking points. 
The subdomain pressure achieves a compromise value of approximately $4.7$--$4.8$, reflecting the tension between the two objectives; the control saturates at the upper bound $\gamma_+ = 30$ in the vicinity of the tracking region to enforce this balance.
In Case~(ii), the resulting pressure field is visibly asymmetric, reflecting the diagonal placement of the two offset points relative to the vertical interface. Correspondingly, the control displays two distinct lobes on either side of the interface, each concentrating injection to drive the local pressure toward the respective point target $p_{\mathscr{P}} = 4$ while partially satisfying the line target $p_{\mathscr{C}} = 6$.
In Case~(iii), the pressure is most tightly controlled along the vertical segment, and the control field exhibits two symmetric lobes flanking the line, with a smaller contribution at the point location.
The subdomain pressure, penalized only weakly, settles near the common target value without active saturation of the control constraint.

\begin{figure}[!h]
     \centering
     \begin{subfigure}[h]{0.35\textwidth}
         \centering
         \includegraphics[width=\textwidth]{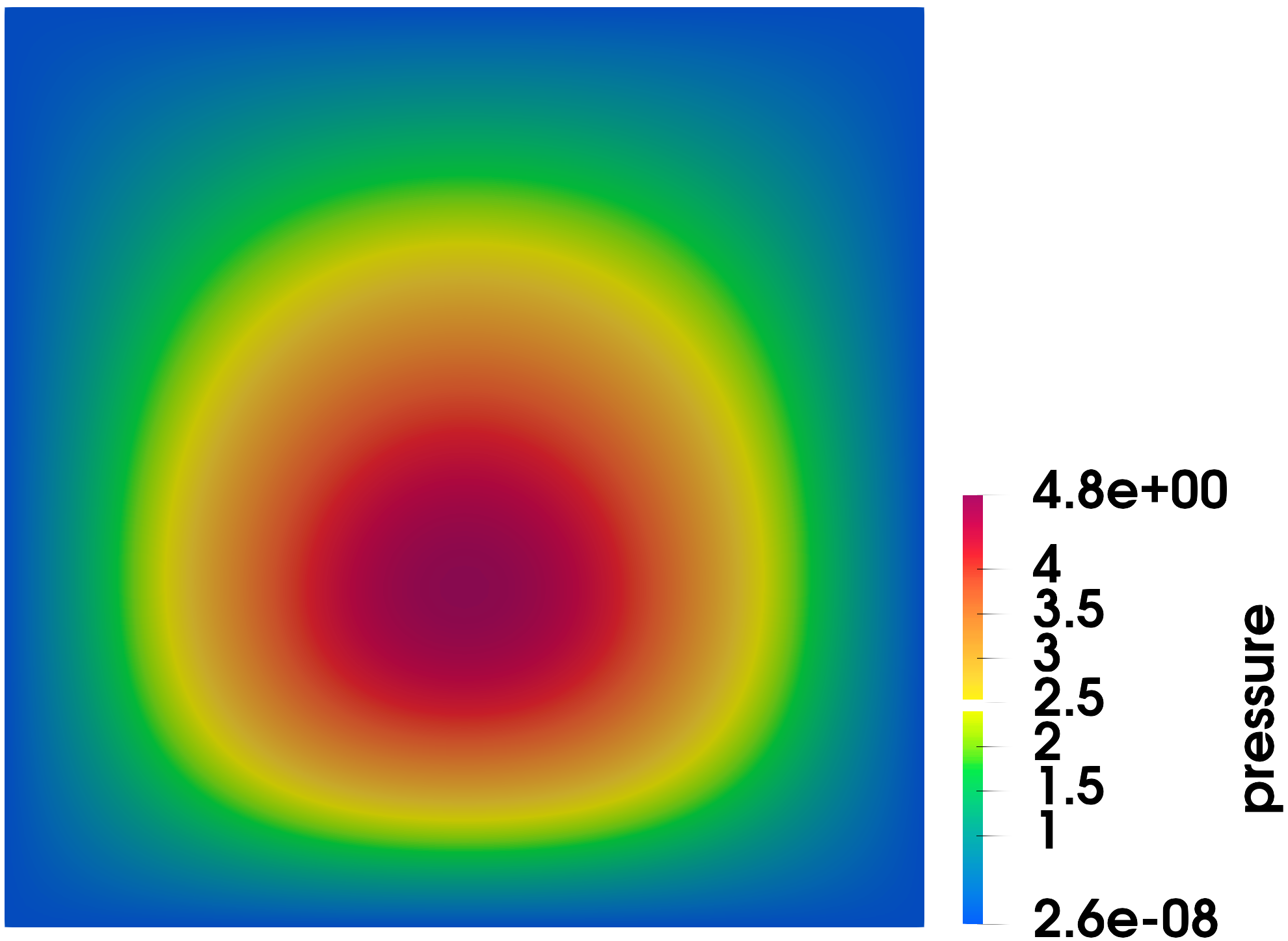}
         \caption{Case i)}
     \end{subfigure}
     \hspace{50pt}
     \begin{subfigure}[h]{0.35\textwidth}
         \centering
         \includegraphics[width=\textwidth]{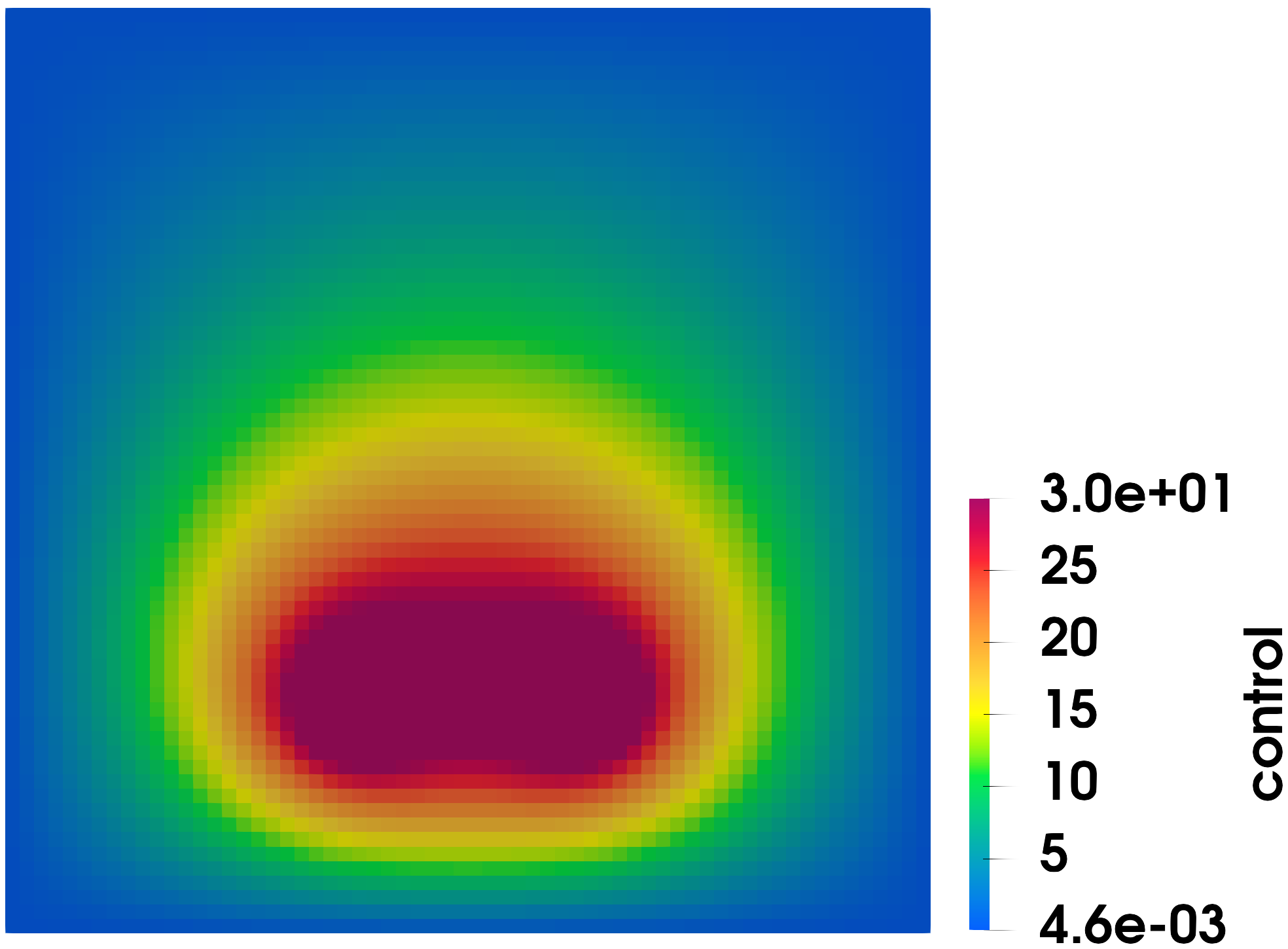}
         \caption{Case i)}
     \end{subfigure}
     \begin{subfigure}[h]{0.35\textwidth}
         \centering
         \includegraphics[width=\textwidth]{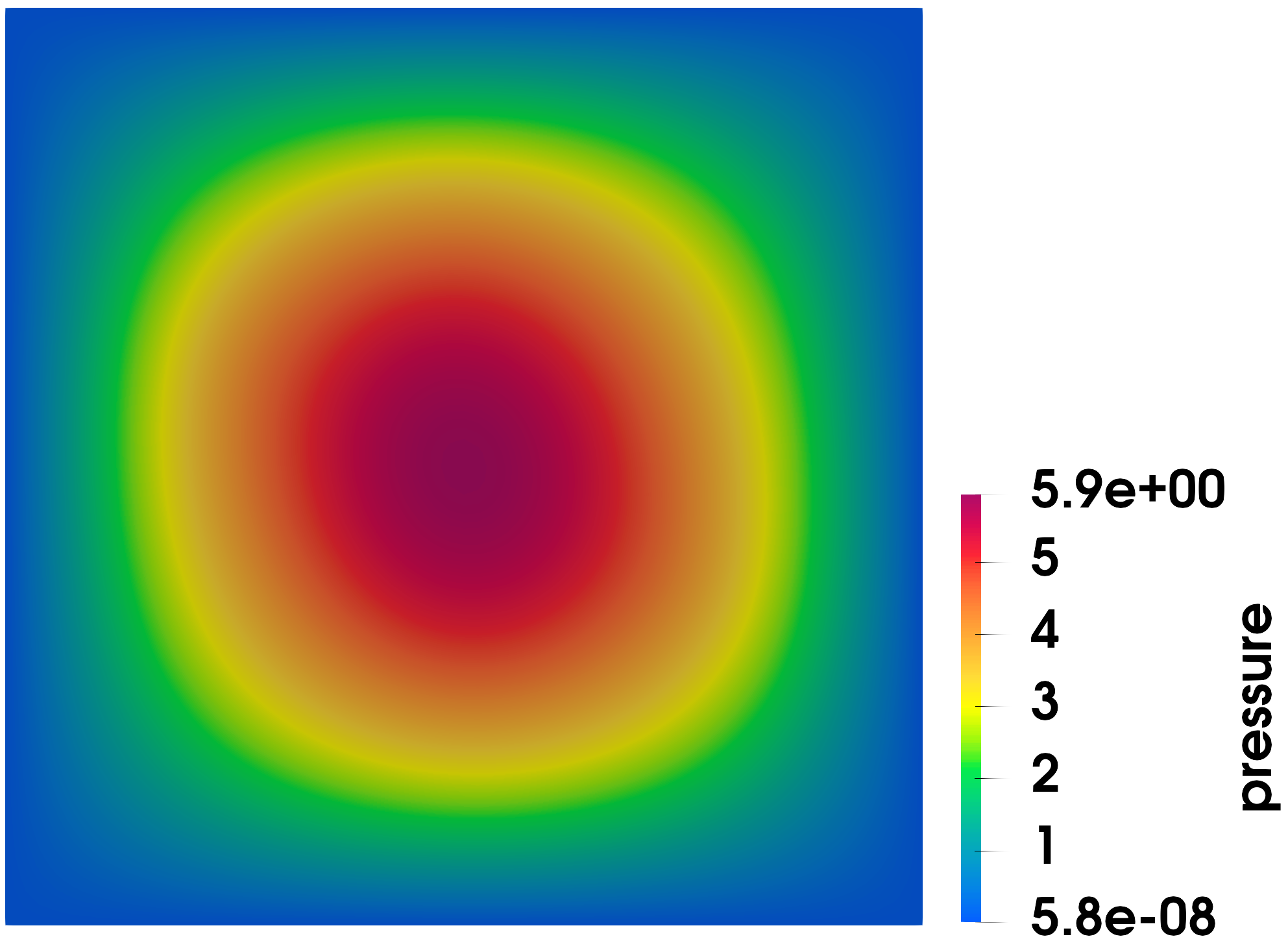}
         \caption{Case ii)}
     \end{subfigure}
     \hspace{50pt}
     \begin{subfigure}[h]{0.35\textwidth}
         \centering
         \includegraphics[width=\textwidth]{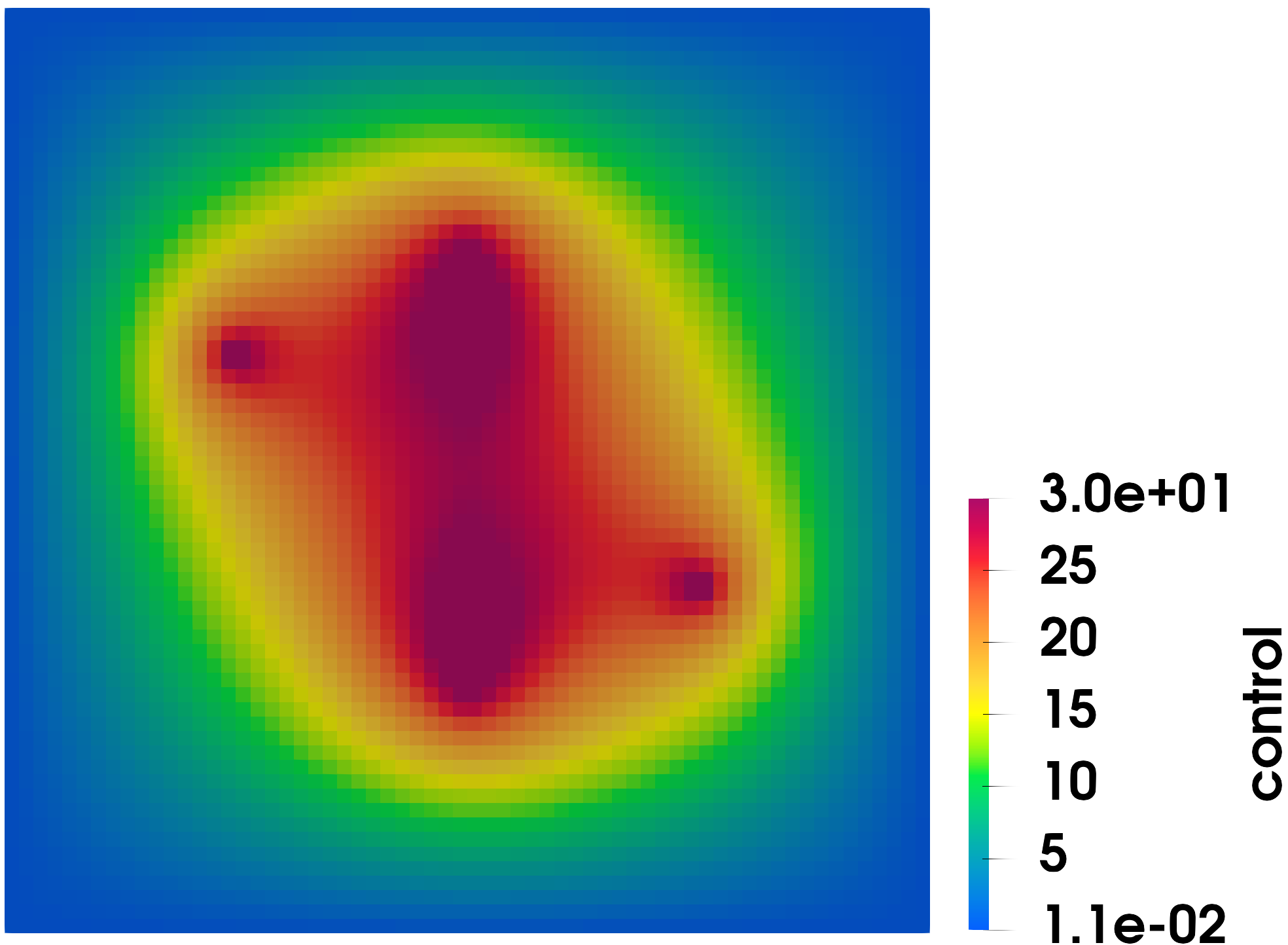}
         \caption{Case ii)}
     \end{subfigure}
     \begin{subfigure}[h]{0.35\textwidth}
         \centering
         \includegraphics[width=\textwidth]{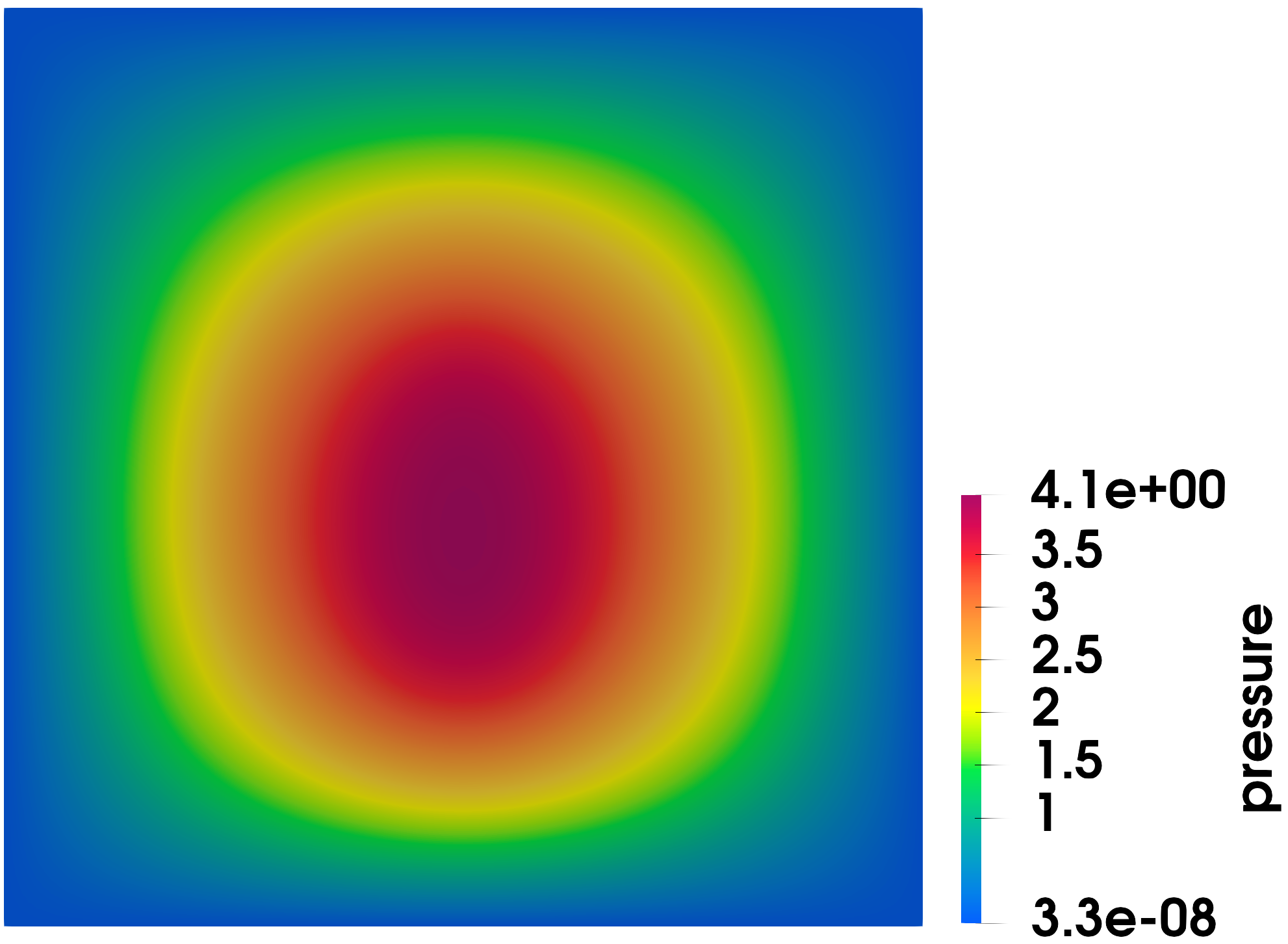}
         \caption{Case iii)}
     \end{subfigure}
     \hspace{50pt}
     \begin{subfigure}[h]{0.35\textwidth}
         \centering
         \includegraphics[width=\textwidth]{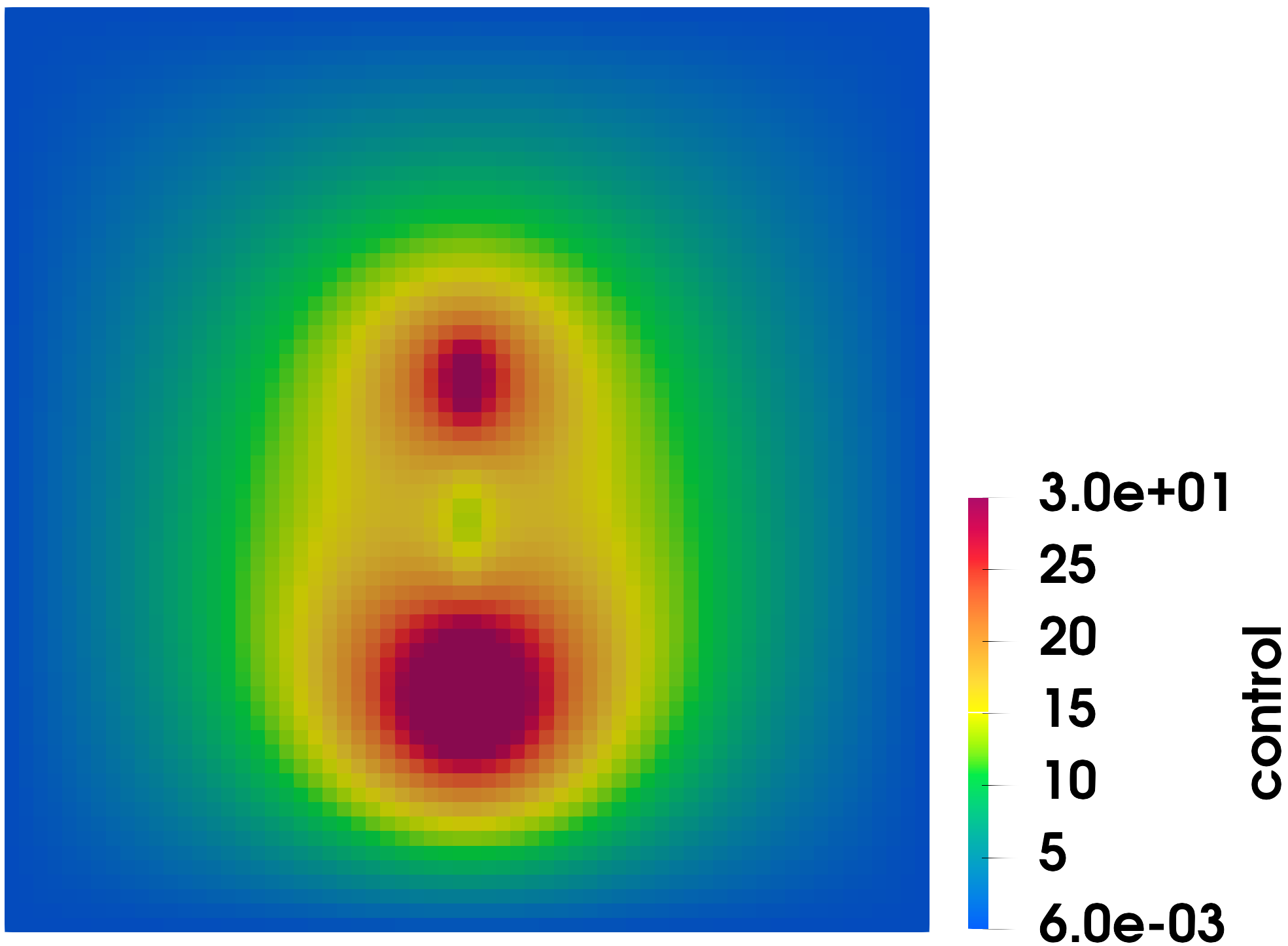}
         \caption{Case iii)}
     \end{subfigure}
        \caption{Example 1. Illustrates the state variable $\bar{p}_h$ and the control variable $\bar{\gamma}_h$ for each different cases. }
        \label{fig:ex1}
\end{figure}

\subsection*{Numerical Verification of Local Mass Conservation}

We verify the local mass conservation property \eqref{eqn:discrete_local_mass_conservation} through Case (iii) of Example 1 numerically.
For each element $T \in \mathcal{T}_h$, define the local residual
\begin{equation*}
  R_T \;:=\; \left|\,\sum_{e \in \partial T} \int_e \hat{q}_{h,T}\,ds
             \;-\; \int_T \bigl(f + \bar{\gamma}_h\bigr)\,d\mathbf{x}\,\right|,
\end{equation*}
where the outward numerical flux on $e \in \partial T$ is given by \eqref{flux}.

Since $f = 0$ and $\bar{\gamma}_h \in W_h$ is piecewise constant on each cell,
\begin{equation*}
  \int_T \bigl(f + \bar{\gamma}_h\bigr)\,d\mathbf{x}
  \;=\; \bar{\gamma}_h\big|_T \cdot |T|.
\end{equation*}

For an interior face $e$ shared by cells $T$ (cell~1) and $T'$ (cell~2),
with outward unit normal $\mathbf{n}$ of $T$, the numerical flux from \eqref{flux} reads
\begin{equation*}
  \int_e \hat{q}_{h,T}\,ds
  \;=\; \int_e \Bigl(
    -\tfrac{1}{2}\bigl(\mathbf{K}\nabla\bar{p}_h\big|_T
                      + \mathbf{K}\nabla\bar{p}_h\big|_{T'}\bigr)\cdot\mathbf{n}
    \;+\; \sigma\,\bigl(\bar{p}_h\big|_T - \bar{p}_h\big|_{T'}\bigr)
  \Bigr)\,ds,
\end{equation*}
where $\mathbf{K} = I$ and $\sigma$ is the SIPG penalty parameter.
For a boundary face $e \subset \partial\Omega$ with homogeneous Dirichlet
condition $\bar{p}_h = 0$,
\begin{equation*}
  \int_e \hat{q}_{h,T}\,ds
  \;=\; \int_e \Bigl(
    -\mathbf{K}\nabla\bar{p}_h\big|_T \cdot \mathbf{n}
    \;+\; \sigma\,\bar{p}_h\big|_T
  \Bigr)\,ds,.
\end{equation*}

\begin{table}[h]
\centering
\setlength{\tabcolsep}{4pt}
\begin{tabular}{c||c|c|c|c|c|c}
\hline
$h$             & $2^{-1}$ & $2^{-2}$ & $2^{-3}$ & $2^{-4}$ & $2^{-5}$ & $2^{-6}$ \\
\hline
$\max_T |R_T|$    & $1.3{\times}10^{-9}$ & $6.9{\times}10^{-11}$
                  & $3.7{\times}10^{-11}$ & $4.3{\times}10^{-13}$ & $2.4{\times}10^{-11}$
                  & $4.0{\times}10^{-14}$ \\
\hline
\end{tabular}
\caption{Maximum local conservation residual $\max_T |R_T|$ across refinement levels.}
\label{tab:lmc}
\end{table}

Table~\ref{tab:lmc} reports $\max_{T \in \mathcal{T}_h} R_T$ for each refinement level.
The residuals remain at machine-precision level on all meshes, confirming that the SIPG discretization satisfies~\eqref{eqn:discrete_local_mass_conservation} exactly, as guaranteed by the theoretical result in Section~\ref{sec:local_mass_conservation}.
\subsection{Example 2. Heterogeneous domain with layers} 
In this subsection, we compare four different cases, which can be found in \cite{jeong2025optimal}.
Let the computational domain $\Omega=(-1,1)^2$.
Consider the permeability tensor $\vK=K\vI$, where $K$ is a piecewise constant, and $\vI$ is the $2\times2$ identity matrix.
The cases we consider are as follows:\[
\text{(i)}\; \vK=5\vI \;\text{in} \;\Omega,\quad
\text{(ii)}\; \vK=\vI \;\text{in} \;\Omega,\quad
\text{(iii)}\; \vK =\left\{ \begin{array}{ll}
    \vI & \text{in}\; \Omega_1 \\
    10\vI & \text{in}\; \Omega_2
\end{array}\right.,\quad
\text{(iv)}\; \vK =\left\{ \begin{array}{ll}
    3\vI & \text{in}\; \Omega_1 \\
    \vI & \text{in}\; \Omega_2
\end{array}\right..
\]
Here $\Omega=\Omega_1\cup\Omega_2$, where
$\Omega_1=\{(x,y)|\;y\leq 0\}$ and $\Omega_2=\{(x,y)|\;y> 0\}$.

\begin{figure}[h!]
     \centering
     \begin{subfigure}[h]{0.2\textwidth}
         \centering
         \begin{tikzpicture}[scale=0.80]
         \draw (0,0) -- (4,0) -- (4,4) -- (0,4) -- (0,0);
         \node at (3.8,3.8) () {$\Omega$};
         \node at (2,2) () {$\vK=5\vI$};
         \draw[dashed, blue!70!black, fill=blue!20, fill opacity=0.7] (1,0.5) 
         rectangle (3,1.5);
         \end{tikzpicture}
         \caption{Case i)}
         \label{fig:ex2_case i)}
     \end{subfigure}
     \hfill
     \begin{subfigure}[h]{0.2\textwidth}
         \centering
           \begin{tikzpicture}[scale=0.80]
          \draw (0,0) -- (4,0) -- (4,4) -- (0,4) -- (0,0);
         \node at (3.8,3.8) () {$\Omega$};
         \node at (2,2) () {$\vK=\vI$};
        \draw[dashed, blue!70!black, fill=blue!20, fill opacity=0.7] (1,0.5) 
         rectangle (3,1.5);
         \end{tikzpicture}
         \caption{Case ii)}
         \label{fig:ex2_case ii)}
     \end{subfigure}
     \hfill
     \begin{subfigure}[h]{0.2\textwidth}
         \centering
           \begin{tikzpicture}[scale=0.80]
           \fill[gray!10!white] (0,0) rectangle (4,2);
         \draw (0,0) -- (4,0) -- (4,4) -- (0,4) -- (0,0);
         \draw (0,2) -- (4,2);
         \node at (3.8,3.8) () {$\Omega$};
         \node at (2,1.75) () {$\vK=\vI$};
         \node at (2,3.75) () {$\vK=10\vI$};
         \node at (0.3,3.8) () {$\Omega_2$};
         \node at (0.3,1.8) () {$\Omega_1$};
        \draw[dashed, blue!70!black, fill=blue!20, fill opacity=0.7] (1,0.5) 
         rectangle (3,1.5);
         \end{tikzpicture}
          \caption{Case iii)}
     \end{subfigure}
     \hfill
     \begin{subfigure}[h]{0.2\textwidth}
         \centering
            \begin{tikzpicture}[scale=0.80]
            \fill[gray!10!white] (0,2) rectangle (4,4);
         \draw (0,0) -- (4,0) -- (4,4) -- (0,4) -- (0,0);
         \draw (0,2) -- (4,2);
         \node at (3.8,3.8) () {$\Omega$};
         \node at (2,1.75) () {$\vK=3\vI$};
         \node at (2,3.75) () {$\vK=\vI$};
         \node at (0.3,3.8) () {$\Omega_2$};
         \node at (0.3,1.8) () {$\Omega_1$};
        \draw[dashed, blue!70!black, fill=blue!20, fill opacity=0.7] (1,0.5) 
         rectangle (3,1.5);
         \end{tikzpicture}
          \caption{Case iv)}
         \label{fig:ex2_case iv)}
     \end{subfigure}
\caption{Example 2. Four different cases.}
\label{fig:ex2_setsup}
\end{figure}
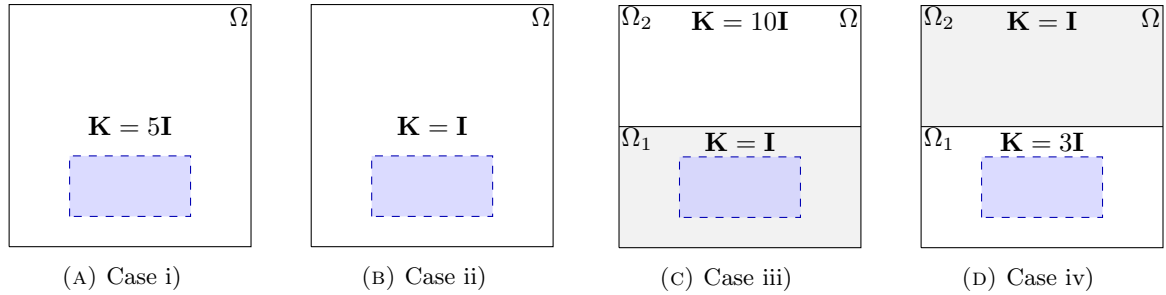

The tracking region is a rectangular subdomain which is bounded by $(-0.5,-0.75)\times(0.5,-0.25)$,.
The control regularization parameter $\beta=1$, the upper bound for the control constraints $\gamma_+=400$, and the weight for the tracking region $w_\mathscr{E}=10^4$, and the desired state for the tracking region $p_d=4.4$.

\begin{table}[h]
\centering
\setlength{\tabcolsep}{3pt}
\begin{tabular}[h!]{c||c|c||c|c||c|c||c|c}
& \multicolumn{2}{c||}{Case i)}  
& \multicolumn{2}{c||}{Case ii)}  
& \multicolumn{2}{c||}{Case iii)}  
& \multicolumn{2}{c}{Case iv)}   \\
$h$ & error & order 
& error & order 
& error & order
& error & order\\
\hline
$2^{-2}$&5.38e+01& -   &2.49e+01& -  &2.91e+01& -  &3.81e+01& - \\
$2^{-3}$&2.91e+01&0.89 &1.66e+01&0.59&1.75e+01&0.73&2.09e+01&0.87\\
$2^{-4}$&1.25e+01&1.22 &7.23e+00&1.20&7.69e+00&1.19&9.51e+00&1.14\\
$2^{-5}$&5.39e+00&1.21 &3.09e+00&1.23&3.32e+00&1.21&4.13e+00&1.20
\end{tabular}
\caption{Example 2. Convergence results for each different cases. }
\label{tab:example2}
\end{table}

Table~\ref{tab:example2} shows convergence results for the four permeability configurations. 
For the homogeneous cases~(i) and~(ii), the errors decay at rates approaching $1.21$--$1.23$ at the finest refinement levels, consistent with the theoretical prediction. 
The heterogeneous cases~(iii) and~(iv), in which the permeability jumps by a factor of $10$ or $3$ across the interface $\{y = 0\}$, yield convergence orders of $1.19$--$1.21$ at the finest levels, matching the homogeneous cases and confirming that the proposed DG discretization handles material discontinuities without degrading the convergence rate. 
In all four cases the tracking region $[-0.5,\,0.5]\times [-0.75,\,-0.25]$ is located entirely within $\Omega_1$, so the permeability contrast directly affects the flow path from the control to the tracking region; the fact that convergence rates are unaffected demonstrates that the SIPG penalty handles the discontinuous coefficient robustly.
These four configurations follow the setup of~\cite{jeong2025optimal}, and the convergence rates obtained here are consistent with those reported therein, providing additional validation of the monolithic method.

\section{Conclusions}
\label{sec:conclusions}

We have developed and analyzed a monolithic discontinuous Galerkin framework for optimal control of Darcy flow with Radon-measure tracking functionals and pointwise control constraints. 
The SIPG discretization couples the state, adjoint, and control variables in a single primal-dual active set iteration, which inherits the local mass conservation property of the underlying DG scheme.

The main analytical difficulty stems from the reduced regularity of the adjoint state, which arises because the tracking functional is supported on a lower-dimensional set rather than on a volume measure. 
We address this in two stages. 
First, we develop approximation theory for the DG Ritz projection measured against the tracking measure itself, exploiting elliptic regularity and a Sobolev embedding to transfer $L^2$-based estimates to the singular-measure norm. 
Second, and more subtly, we control the cross-coupling between the discrete adjoint error and the discrete control error, which cannot be treated directly without introducing a circular argument. 
The resolution is to introduce an intermediate discrete adjoint driven by the continuous optimal state, which decouples the two error quantities and allows each to be bounded independently. 
Theorem~\ref{thm:error} assembles these estimates into convergence rates that reflect the reduced adjoint regularity through the elliptic regularity index $\alpha$.

Numerical experiments on heterogeneous porous media confirm the predicted rates and demonstrate local mass conservation of the computed flow field. 
The numerical results also suggest that the state error converges at the sharper rate $\mathcal{O}(h^{2\alpha})$, improving upon the $\mathcal{O}(h^{\min\{\theta,1-\epsilon\}})$ bound established in Theorem~\ref{thm:error}, and closing this gap theoretically is left for future work.
Additional directions include extending the framework to time-dependent Darcy flow, incorporating more general permeability tensors, and developing adaptive mesh strategies that exploit the local structure of Radon-measure tracking functionals to recover improved convergence near the support of the measure.

\section*{Acknowledgements}
The author S. Lee was supported by the National Research Foundation of Korea (NRF) grant funded by the Korea government (MSIT) (RS-2026-25505849).
\bibliographystyle{plain}
\bibliography{biblio}

\end{document}